\documentclass[12pt]{amsart}
\usepackage{latexsym}
\usepackage{amssymb}
\usepackage{amsmath}
\usepackage{color}

\newtheorem{Thm}[subsection]{Theorem}
\newtheorem{Lemma}[subsection]{Lemma}

\newtheorem{Def}[subsection]{Definition}
\newtheorem{Cor}[subsection]{Corollary}

\numberwithin{equation}{section}

\newcommand{\ben}{\begin{enumerate}}
\newcommand{\een}{\end{enumerate}}

\newcommand{\bec}{\begin{center}}
\newcommand{\eec}{\end{center}}

\newcommand{\beq}{\begin{equation}}
\newcommand{\eeq}{\end{equation}}

\newcommand{\bdm}{\begin{displaymath}}
\newcommand{\edm}{\end{displaymath}}

\newcommand{\R}{\mathbb{R}}

\newcommand{\N}{\mathbb{N}}

\newenvironment{proofof}[1]{{\sc Proof of #1}}{\quad\lower0.05cm\hbox{$\square$}\medskip}

\title[Schr\"{o}dinger Operators in Uniform Domains]
{Positive Solutions and Harmonic Measure for Schr\"{o}dinger Operators\\ in Uniform Domains}

\author{Michael W. Frazier}
\address{Mathematics Department, University of Tennessee, Knoxville,
Tennessee 37922, USA} \email{mfrazie3@utk.edu}

\author{Igor E. Verbitsky}
\address{Department of Mathematics, University of Missouri, Columbia, Missouri 65211, USA}
\email{verbitskyi@missouri.edu}



\begin{document}

\begin{abstract} We give bilateral pointwise estimates for positive solutions of the equation 
\begin{equation*} 
\left\{ \begin{aligned}
-\triangle  u & = \omega u \, \,&  &  \mbox{in} \, \,  \Omega, \quad  u \ge 0,  \\
u & = f \, \, &  &\mbox{on} \, \,  \partial \Omega ,
\end{aligned}
\right.  
\end{equation*} 
in a bounded uniform domain $\Omega\subset \R^n$, where $\omega$ is a locally finite Borel 
measure in $\Omega$, and $f\ge 0$ is integrable with respect to harmonic measure $d H^{x}$ on $\partial\Omega$. 

We also give sufficient 
and matching necessary conditions for the existence of a positive solution in terms of the exponential 
integrability of $M^{*} (m \omega)(z)=\int_\Omega M(x, z) m(x)\, d \omega (x)$ on $\partial\Omega$   
with respect to $f \, d H^{x_0}$, where $M(x, \cdot)$ is Martin's function 
with pole at $x_0\in \Omega, m(x)=\min (1, G(x, x_0))$, and $G$ is Green's function. 

These results give bilateral bounds for the harmonic measure associated with the Schr\"{o}dinger operator $-\triangle  - \omega $ on $\Omega$, 
and in the case $f=1$,  a criterion for the existence of the gauge function. Applications to elliptic equations of Riccati type with quadratic growth in the gradient 
are given. 
\end{abstract}

\maketitle \vfill

\eject

\tableofcontents

\section{Introduction}

Let $\Omega \subset \R^n$  ($n \geq 3$) be a nonempty, connected, open set  (a domain).  
It is called a non-tangentially accessible (NTA) domain if it is bounded, and satisfies  both the interior  and exterior corkscrew conditions, and the Harnack chain condition  (\cite{JK}). For instance, any bounded Lipschitz domain is an NTA domain. The exterior corkscrew condition yields that any NTA domain is regular  (in the sense of Wiener). 

More generally,  a uniform domain 
is defined as a bounded domain which 
satisfies  the interior corkscrew condition and the Harnack chain condition. 
Uniform domains satisfy the local (or uniform) boundary Harnack principle  (\cite{Aik}; see  \cite{An1}, \cite{JK} for Lipschitz and NTA domains). However, they are not necessarily regular. 
 Our main 
results hold for bounded uniform domains, and the regularity of $\Omega$ is not used below. 

In \cite{Ken}, a slightly more general version of an NTA domain $\Omega$ is defined  as a uniform domain  
 of class $\mathcal{S}$  (Definition 1.1.20), i.e., satisfying the volume density condition, which ensures that $\Omega$ is a regular domain. 
  Most of our results, including Theorem \ref{mainufest} and Theorem \ref{gaugecrit} below, hold in this setup for uniformly  elliptic operators in divergence form  
 $\mathcal{L} = \text{div} (A \nabla \cdot)$, with bounded measurable symmetric $A$, 
 in place of the Laplacian $\triangle$,  as in \cite{JK}, p. 138 and \cite{Ken}, Sec. 1.3. The same  class of operators  $\mathcal{L}$ in uniform domains with Ahlfors regular boundary  
 can be covered as well (see \cite{Zha}).

 In this paper, for simplicity,  we consider mostly the case $n\ge 3$.  In two dimensions, our results hold if $\Omega$  is a bounded finitely connected domain in $\R^2$, in particular, a bounded Lipschitz domain  (see 
 \cite{CZ}, Theorem 6.23; \cite{Han1}, Remark 3.5). 

 Let $\omega$ be a locally finite Borel measure on $\Omega$ and let $f$ be a non-negative Borel measurable function on $\partial \Omega$.  We consider the equation
\begin{equation}\label{ufeqn} 
\left\{ \begin{aligned}
-\triangle  u & = \omega u \, \,&  &  \mbox{in} \, \,  \Omega, \quad  u \ge 0,  \\
u & = f \, \, &  &\mbox{on} \, \,  \partial \Omega.
\end{aligned}
\right.  
\end{equation} 
Solutions of \eqref{ufeqn} are understood either in the potential theoretic sense, or $d \omega$-a.e. 
The precise definitions are discussed in \S \ref{sec2} below.
In the case of $C^2$ domains, or bounded Lipschitz domains $\Omega$, they 
coincide with ``very weak'' solutions in the sense of Brezis (see \cite{BCMR}, \cite{FV2}, \cite{MV}, Sec. 1.2, \cite{MR}).

 Let $\Omega \subseteq \R^n$ be a bounded domain with  Green's function 
$G(x,y)$; then $G$ is symmetric and strictly positive on $\Omega \times \Omega$. For a 
Borel measure $\nu$ on $\Omega$, 
\begin{equation} \label{Greenop}
G \nu (x) = \int_{\Omega} G(x,y) \, d\nu(y), \qquad x\in\Omega,
\end{equation}
is Green's operator. We call $G \nu$ the Green's potential if $G \nu\not\equiv +\infty$. 
  For a Borel measurable function $f$ on $\partial \Omega$, 
define the harmonic extension $Pf$ of $f$ into $\Omega$  (the generalized solution to the Dirichlet problem) by 
\begin{equation} \label{harm-rep}
Pf (x) = \int_{\partial \Omega}  f(z) \,  dH^x (z), \qquad x \in \Omega,  
\end{equation}
where $dH^x$ is the harmonic measure at $x$, if the integral in \eqref{harm-rep} exists.  

A solution $u$ to \eqref{ufeqn} satisfies, formally, the equation 
\begin{equation} \label{form-sol}
u (x) = G(u \omega) (x) + Pf (x), \qquad x \in \Omega .
\end{equation}
We remark that if $u\not\equiv +\infty$ satisfies \eqref{form-sol}, 
then it is a superharmonic function in $\Omega$, and  $Pf$ is its greatest harmonic minorant. 
In particular, $u\in L^1_{loc}(\Omega, dx)\cap L^1_{loc} (\Omega, d\omega)$, and  $u< +\infty$ q.e., that is, quasi-everywhere with respect to the Green capacity, see \cite{AG}, \cite{Lan}.  

We also consider more general equations with an arbitrary positive harmonic function $h$ in place of $Pf$ (see 
\S \ref{sec2} and \S \ref{sec3}), when irregular boundary points may come into play.

For an appropriate function $g$ on $\Omega$, we define 
\begin{equation} \label{defT}
  T g (x) = G(g \omega)(x) = \int_{\Omega} G(x,y) \, g(y) \, d \omega(y) , 
\end{equation}
for $x \in \Omega$, so that equation \eqref{form-sol} becomes $(I-T)u = Pf$, with formal solution 
\begin{equation} \label{ufsolndef}
  u_f = \sum_{j=0}^{\infty} T^j (Pf) . 
\end{equation}
This \textit{minimal} solution $u_f$ of \eqref{ufeqn} satisfies 
\begin{equation} \label{minsolnform}
u_f (x) = G(u_f \omega) (x) + Pf (x), \qquad x \in \Omega , 
\end{equation}
if $u_f\not\equiv +\infty$. 
Under conditions which guarantee the finiteness of the right side of equation (\ref{ufsolndef}) (see Theorem \ref{mainufest}  and Theorem \ref{gaugecrit}), we will see that $u_f$ defined by (\ref{ufsolndef}) gives a (generalized) solution of (\ref{ufeqn}). 

It was shown in \cite{FV3}, Lemma 2.5, that the following are equivalent: for $\beta >0$, 
\begin{equation} \label{normTless1}
T \,\, \mbox{is bounded on} \,\, L^2 (\Omega, \omega) \,\, \mbox{with} \,\, \Vert T \Vert = \Vert T \Vert_{L^2 (\Omega, \omega) \rightarrow L^2 (\Omega, \omega)} \le \beta^2 
\end{equation}
and
\begin{equation} \label{equivnormTless1}
\Vert \varphi \Vert_{L^2 (\omega)} \leq \beta \, \Vert \nabla \varphi \Vert_{L^2 (dx)}, \,\, \mbox{for all} \,\, \varphi \in C^{\infty}_0 (\Omega) .
\end{equation}

Our results are expressed in terms of the \textit{Martin kernel} $M(x,z)$. In a bounded uniform domain $\Omega\subset\R^n$, the Martin boundary $\triangle$ is homeomorphic to the Euclidean boundary $\partial \Omega$  (\cite{Aik}, Corollary 3;  see \cite{HW}, \cite{AG}, for a bounded Lipschitz domain, and  \cite{JK}, \cite{Ken} 
 for an NTA domain.) Martin's kernel, defined with respect to a reference point $x_0 \in \Omega$, is given by 
\begin{equation} \label{martin-K}
M(x, z)=\lim_{y\to z, \, \,  y\in \Omega} \frac{G(x, y)}{G(x_0, y)}, \quad x \in \Omega, \, \, z \in \partial \Omega,
\end{equation}
where the limit exists, and is a minimal harmonic function in $x \in \Omega$.  We will see in \S \ref{sec2} that 
\begin{equation} \label{harmmeasiden}
 dH^x (z) = M(x,z) \, dH^{x_0} (z) ,    \qquad (x, z) \in \Omega\times\partial \Omega ,   
\end{equation}
for uniform domains (see \cite{HW}, p. 519; \cite{CZ}, p. 137 for Lipschitz domains;  
 \cite{JK}, pp. 104, 115 for NTA domains). Combining  \eqref{harm-rep} and \eqref{harmmeasiden} yields
\begin{equation} \label{pf-rep-martin}
Pf (x) = \int_{\partial \Omega}  M(x,z) \, f(z)  \, dH^{x_0} (z), \quad x \in \Omega, 
\end{equation}
for Borel measurable $f\ge 0$, whenever the integral exists. Hence, \eqref{ufsolndef} yields
\begin{equation} \label{ufrepresent}
  u_f (x) =  \int_{\partial \Omega} \sum_{j=0}^{\infty} T^j  M(\cdot,z) (x)\, f(z)  \, dH^{x_0} (z), \quad x \in \Omega. 
	\end{equation}
We define 
\begin{equation} \label{Martin-schro}
\mathcal{M}(x, z)= \sum_{j=0}^{\infty} T^j M(\cdot, z) (x), \qquad (x, z) \in \Omega\times\partial \Omega , 
\end{equation}
and  
\begin{equation} \label{harm-schro}
d \mathcal{H}^x(z) =\mathcal{M}(x, z) \,  dH^{x_0} (z), \qquad (x, z) \in \Omega\times\partial \Omega. 
\end{equation}
Then \eqref{ufrepresent} gives 
\begin{equation} \label{ufrepresent2}
\begin{aligned}
 u_f(x) &=  \int_{\partial \Omega} \mathcal{M}(x,z)\, f(z)  \, dH^{x_0} (z) \\ & = \int_{\partial\Omega} f(z) \, d \mathcal{H}^x(z), \qquad x\in \Omega.
 \end{aligned}
\end{equation}

Comparing this last equation with equation \eqref{harm-rep}, we see that $ d \mathcal{H}^x$ is harmonic measure for the Schr\"{o}dinger operator $-\triangle -\omega$. 

By \eqref{Martin-schro}, 
\begin{align*} \mathcal{M}(x,z) &= M(x,z) + \sum_{j=1}^{\infty} T^j M (\cdot, z) (x) \\
& = M(x,z) + T \mathcal{M}(\cdot, z)(x)\\
& = M(x,z) + G (\mathcal{M}(\cdot, z) \omega)(x).   
\end{align*}
Hence $\mathcal{M}(x, z)$ is a superharmonic function of $x \in \Omega$, and $M(x,z)$ is its greatest harmonic minorant,  for every $z\in \partial \Omega$, provided $M(\cdot,z)\not\equiv\infty$. In fact, $\mathcal{M}(\cdot, z)$ is $\omega$-harmonic, i.e., it satisfies the Schr\"odinger equation 
$-\triangle u=\omega \, u$ in $\Omega$.

Notice that $\mathcal{H}^x$ defined by \eqref{harm-schro} is not a probability measure on $\partial \Omega$ unless $\omega=0$. Letting $f\equiv 1$ on $\partial\Omega$, we see by \eqref{ufrepresent2} that $\mathcal{H}^x$ is a finite measure on $\partial \Omega$ if and only if 
$u_1(x)<\infty$, where $u_1$ is the so-called gauge function defined by \eqref{gauge-def} below (see Corollary \ref{cor} for conditions under which $u_1<\infty$ $d\omega$-a.e.).

We remark that for the normalized version of  $\mathcal{M}(x, z)$ defined by 
\[
 \widetilde{\mathcal{M}}(x, z)=\frac{\mathcal{M}(x, z)}{\mathcal{M}(x_0, z)}, \qquad (x, z) \in \Omega\times\partial \Omega , 
\] 
where $x_0\in \Omega$ is to be chosen so that $\mathcal{M}(x_0, z)<\infty$ for every $z \in \partial \Omega$, we have 
\[
d \mathcal{H}^x(z) =  \widetilde{\mathcal{M}}(x, z)
 \,  
d \mathcal{H}^{x_0} (z), \qquad (x, z) \in \Omega\times\partial \Omega , 
\] 
which is analogous to \eqref{harmmeasiden}.  Obviously, $\widetilde{\mathcal{M}}(x_0, z)=1$, as for the unperturbed Martin's kernel $M(x,z)$. Moreover, \textit{formally} we have 
\[
 \widetilde{\mathcal{M}}(x, z)=\lim_{y\to z, \, y \in \Omega}\frac{\mathcal{G}(x, y)}{\mathcal{G}(x_0, y)}, 
\qquad (x, z) \in \Omega\times\partial \Omega ,
\]
where $\mathcal{G}(x, y)$ is the minimal Green's function associated with  the 
Schr\"{o}dinger operator $-\triangle -\omega$  (see  \cite{FNV}).  Thus,  
$\widetilde{\mathcal{M}}(x, z)$ 
serves the role of the (normalized) Martin 
kernel associated with  the Schr\"{o}dinger operator $-\triangle -\omega$.

Nevertheless, we prefer to use the  kernel $\mathcal{M}(x, z)$, since it does not exclude the case  $\mathcal{M}(x_0, z)=\infty$, and is more  convenient in applications. Pointwise estimates of 
$\widetilde{\mathcal{M}}(x, z)$  are deduced easily from the estimates 
of $\mathcal{M}(x, z)$ discussed below.

Our bilateral estimates  of $\mathcal{M}(x, z)$ (see \eqref{upperTM} and \eqref{lowerTM} below)    
are stated in terms of exponentials:  
\begin{equation} \label{exp-term}
\begin{aligned}
 M(x,z) & \, e^{\int_{\Omega} G(x,y) \, \frac{M(y,z)}{M(x,z)} \, d \omega (y)} \le \mathcal{M}(x, z)
 \\ & \le   M(x,z) \, e^{C\int_{\Omega} G(x,y) \, \frac{M(y,z)}{M(x,z)} \, d \omega (y)} , 
  \end{aligned}
 \end{equation}
for all $(x, z) \in \Omega\times\partial \Omega$,
with an appropriate constant $C>0$.  We remark that 
\[
\mathcal{M}(x, z) = U(x, z) \, M(x,z), \qquad (x, z) \in \Omega\times\partial \Omega ,
\]
where
\begin{equation} \label{cond-gauge-def}
U(x, z) = 1+ \frac{1}{M(x,z)}\sum_{j=1}^\infty  T^j M(\cdot, z) (x) , \qquad (x, z) \in \Omega\times\partial \Omega ,
\end{equation}
is the so-called \textit{conditional gauge} (\cite{CZ}, Sec. 4.3). 

From \eqref{exp-term} it is immediate that 
\begin{equation} \label{exp-gauge}
 e^{\int_{\Omega} G(x,y) \, \frac{M(y,z)}{M(x,z)} \, d \omega (y)} \le U(x, z)
 \le   e^{C\int_{\Omega} G(x,y) \, \frac{M(y,z)}{M(x,z)} \, d \omega (y)} , 
 \end{equation} 
 for all $(x, z) \in \Omega\times\partial \Omega$. We emphasize that in the exponents of \eqref{exp-gauge} 
we only use  the \textit{first term} in the sum on the right-hand side of \eqref{cond-gauge-def}.

A probabilistic definition of the 
conditional gauge in the case  $d\omega =q \, dx$ ($q \in L^1_{loc} (\Omega)$) is provided by 
\[
U(x, z)= {E}_z^{x} \left[ e^{\int_0^{\zeta} q(X_t) \, dt}\right],   \qquad (x, z) \in \Omega\times\partial \Omega ,
 \]
where $X_t$ is a  path  of the Brownian motion (properly scaled to 
replace   $\frac{1}{2} \triangle$ used in the probabilistic literature with  $\triangle$) starting at $x$,   ${E}_z^{x}$ is the conditional expectation conditioned on the event that $X_t$  exits  $\Omega$ at $z\in \partial\Omega$, and $\zeta$ is the
time when $X_t$  first hits $z$. Properties of the conditional gauge for  potentials $q$ in Kato's class 
in a bounded Lipschitz domain $\Omega$  are discussed in \cite{CZ}, Ch. 7; in particular, $U(x, z)\approx 1$ if  $U(x, z)\not\equiv +\infty$. 

For general $\omega \ge 0$, we clearly have $U(x, z)\ge 1$, but $U(x, z)$ 
  is  no longer uniformly bounded from above, even if $U(x, z) \not\equiv +\infty$ and $||T||<1$. Consequently,  
  the so-called Conditional Gauge Theorem fails in this setup. 
  
\begin{Thm} \label{mainufest}  Let $\Omega\subset \R^n$ be a bounded uniform domain, $\omega$ a locally finite Borel measure on $\Omega$, and $f \geq 0$ a Borel measurable function on $\partial \Omega$.  

\vspace{0.1in}

(A)  If $\Vert T \Vert <1$ (equvalently, (\ref{equivnormTless1}) holds with $\beta <1$), then there exists a positive constant $C$ depending only on $\Omega$ and $\Vert T \Vert$ such that 
\begin{equation} \label{ptwiseupperbnd}
u_f (x) \leq  \int_{\partial \Omega} e^{C \int_{\Omega} G(x,y) \frac{M(y,z)}{M(x,z)} d \omega (y)} f(z) \, dH^x(z), \quad x \in \Omega .
\end{equation}

\vspace{0.1in}

(B) If $u$ is a positive  solution of \eqref{ufeqn}, then $\Vert T \Vert \leq 1$ (equivalently, (\ref{equivnormTless1}) holds for some $\beta \leq 1$) and 
\begin{equation} \label{ptwiselowerbnd}
  u (x) \geq  \int_{\partial \Omega} e^{\int_{\Omega} G(x,y) \frac{M(y,z)}{M(x,z)} d \omega (y)} f(z) \, dH^x(z), \quad x \in \Omega . 
\end{equation}
\end{Thm}

In view of \eqref{ufrepresent2}, Theorem \ref{mainufest} gives estimates for the Schr\"{o}dinger harmonic measure $d \mathcal{H}^x$ in terms of the harmonic measure $dH^{x}$ for the Laplacian.

The solution $u_1$ of (\ref{ufeqn}),  in the case where $f$ is identically $1$ on $\partial \Omega$, is called the (Feynman-Kac) \textit{gauge}:
\begin{equation} \label{gauge-def}
  u_1 = 1 + \sum_{j=1}^{\infty} T^j 1 , 
\end{equation}
provided $u_1\not\equiv +\infty$. An equivalent  probabilistic interpretation of the gauge when  $d\omega =q(x) \, dx$ ($q \in L^1_{loc} (\Omega)$, $q \ge 0$) is given by (see \cite{CZ}, Sec. 4.3)
\[
u_1(x)= {E}^{x} \left[ e^{\int_0^{\tau_\Omega} q(X_t) \, dt}\right],   \qquad x \in \Omega ,  
\]
where $X_t$ is the Brownian path  (properly scaled as above) starting at $x$,  ${E}^{x} $ is the expectation operator,  and $\tau_\Omega$ is the exit time from $\Omega$. Notice that $u_1$ given by \eqref{gauge-def} is related 
to the conditional gauge $U(x,z)$ defined by  \eqref{cond-gauge-def} via 
the equation 
\[
u_1(x)= \int_{\partial\Omega}  U(x, z) \, dH^x(z), \qquad x \in \Omega . 
\]
In particular, 
\[
\inf_{z \in \partial \Omega}  U(x, z) \le u_1(x ) \le\sup_{z \in \partial \Omega}   U(x, z), \qquad x \in \Omega . 
\]
 
  The following theorem gives sufficient and matching necessary criteria for the existence of $u_f$.  For Martin's kernel $M(x,z)$, we define the adjoint operator $M^*$ for a Borel measure $\mu$ on $\Omega$ by 
\begin{equation} \label{defMstar}
M^* \mu (z) = \int_{\Omega} M(x,z) \, d \mu (x), \quad \mbox{for} \,\, z \in \partial \Omega.
\end{equation}
The role of $M^*$ in the following theorem is analogous to the role of the balayage operator $P^*$ in \cite{FV2} for $C^{1,1}$ domains $\Omega$, where  
all integrals over $\partial \Omega$ are taken with respect to surface area in place of harmonic measure.

\begin{Thm} \label{gaugecrit} Suppose $ \Omega \subset \R^n$ is a bounded uniform domain, 
$\omega$ is a locally finite Borel measure on $\Omega$, and $f \geq 0$ ($f$ not a.e. $0$ with respect to harmonic measure) is a Borel measurable function on $\partial \Omega$.  Let $x_0 \in \Omega$ be the reference point in the definition of Martin's kernel.   Let $m(x) = \min (1, G(x,x_0))$.  

\vspace{0.1in}

(A)  There exists $C>0$ ($C$ depending only on $\Omega$ and $\Vert T \Vert$) such that if $\Vert T \Vert <1$ (equivalently, (\ref{equivnormTless1}) holds with $\beta <1$) and 
\begin{equation} \label{martincritsuff}
 \int_{\partial \Omega} e^{C M^* (m \omega)} \,  f \, dH^{x_0} < \infty , 
\end{equation}
then $u_f \in L^1_{loc} (\Omega, dx)$.  

\vspace{0.1in}

(B) If $u_f \in L^1_{loc} (\Omega, dx) $, then $\Vert T \Vert \leq 1$ and  
\begin{equation} \label{martincritnec}
 \int_{\partial \Omega} e^{M^* (m \omega)} \,  f  \, dH^{x_0} < \infty . 
\end{equation}

\end{Thm} 

\bigskip

\noindent {\bf Remark.}
 More general results for equation (\ref{form-sol}) with an arbitrary positive harmonic function $h$ in place of $Pf$, in terms of 
 Martin's representation, are given in Theorem \ref{mainu-harm} and Theorem \ref{u_h-exist}  below. 

\bigskip

For $C^{1,1}$ domains $\Omega$ and absolutely continuous $\omega$, Theorem \ref{mainufest} and an analogue of Theorem \ref{gaugecrit} were proved in the special case $f=1$ in \cite{FV2}, Theorem 1.2. To see this observation, note that for a $C^{1,1}$ domain, $M(x, z) = P(x,z)/P(x_0, z)$, by \eqref{harmmeasiden}, which shows that inequalities (1.12) and (1.14) in \cite{FV2} follow from Theorem \ref{mainufest} above.  To see that (1.10) and (1.13) in \cite{FV2} follow from Theorem \ref{gaugecrit}, choose $x_0$ with dist$(x_0, \partial \Omega) > \delta$, where  $0< \delta < \mbox{diam} ( \Omega) /2$, so that $P(x_0, z)$ is equivalent to a constant depending only on $\Omega$.  An extension  to the case of uniform domains of the criteria in \cite{FV2} for the existence of the nontrivial gauge ($u_1\not\equiv +\infty$) is provided by the following corollary.

\begin{Cor} \label{cor} Suppose $ \Omega \subset \R^n$ is a bounded uniform domain, and $\omega$ is a locally finite Borel measure on $\Omega$.  Let $x_0 \in \Omega$ be the reference point in the definition of Martin's kernel, and $m(x) = \min (1, G(x,x_0))$.  

\vspace{0.1in}

(A)  There exists $C>0$ ($C$ depending only on $\Omega$ and $\Vert T \Vert$) such that if $\Vert T \Vert <1$ and 
\begin{equation} \label{martincritsuff-g}
 \int_{\partial \Omega} e^{C M^* (m \omega)} \,   dH^{x_0} < \infty , 
\end{equation}
then the gauge  $u_1$ is nontrivial.  

\vspace{0.1in}

(B) If the gauge $u_1$ is nontrivial, then $\Vert T \Vert \leq 1$ and 
\begin{equation} \label{martincritnec-g}
 \int_{\partial \Omega} e^{M^* (m \omega)} \, dH^{x_0} < \infty . 
\end{equation}

\end{Cor} 

 As an application of Corollary \ref{cor}, we consider elliptic equations of Riccati type with quadratic growth in the gradient, 
\begin{equation}\label{nonlineareqn-1} 
\left\{
\begin{aligned}
-\triangle v & = |\nabla v|\,^2 + \omega \, \, & \mbox{in} \, \,  \Omega \\
v & = 0 \quad & \mbox{on} \, \,  \partial \Omega 
\end{aligned}
\right.  
\end{equation} 
for locally finite Borel measures $\omega$, 
in bounded uniform domains  $\Omega \subset \R^n$.  Although \eqref{nonlineareqn-1} is formally related to equation \eqref{ufeqn} with $f=1$ by the relation $v = \log u$, it is well-known that this formal relation is not sufficient to guarantee equivalence of the two equations (see \S4).  Nevertheless we obtain the following result.

\begin{Thm}\label{riccatithm}  Suppose $\Omega \subset \R^n$ is a bounded uniform domain, and  $\omega$ is a locally finite Borel measure in $\Omega$.    

(A) Suppose $||T||<1$, or equivalently  (\ref{equivnormTless1}) holds with $\beta<1$, and (\ref{martincritsuff-g}) holds with a  large enough constant $C>0$ (depending only on $\Omega$ and $||T||$). Then $v= \log u_1 \in W^{1,2}_{loc} (\Omega)$ is a weak solution of (\ref{nonlineareqn-1}).  
 
(B) Conversely, if (\ref{nonlineareqn-1}) has a weak solution $v\in W^{1,2}_{loc} (\Omega)$, then 
$u=e^v$ is a supersolution to 
 (\ref{ufeqn}) with $f=1$, i.e., $u\ge G(\omega u) +1$. 
Moreover,  $||T||\le 1$, 
or equivalently (\ref{equivnormTless1}) 
holds with $\beta = 1$, and (\ref{martincritnec-g}) holds.  
\end{Thm}

\vspace{0.1in}

\noindent{\bf Remarks.} 1. In Theorem \ref{gaugecrit}, $u_f \in L^1_{loc} (\Omega, dx)$ actually yields 
$u_f \in L^1(\Omega, m dx)\cap L^1(\Omega, m d \omega)$, or equivalently 
$G  (u_f \omega) \not\equiv+\infty$. 

 2. For bounded Lipschitz domains $\Omega$, $u_1$ is a ``very weak'' solution in the sense of \cite{MR}. More precisely, 
$u=u_1 -1 $ is a ``very weak'' solution 
to $-\triangle u = \omega u + \omega$ with $u=0$ on $\partial \Omega$. Here one can use  $\phi_1$ in place 
of $m$, where $\phi_1$ is the first eigenfunction of the Dirichlet Laplacian in $\Omega$ (see 
\cite{AAC}, Lemma 3.2). Then 
$u_1 \in L^1(\Omega, \phi_1 dx)$ and $\int_\Omega \phi_1 \, d \omega<+\infty$.

3. Our main results for uniform domains $\Omega$ are based on the exponential bounds for  Green's function $\mathcal{G}(x,y)$ (see Theorem \ref{FNVTheorem} below) obtained in \cite{FNV}. Here $\mathcal{G}(x,y)$ is  the kernel of  
 the operator $(I-T)^{-1}$ 
defined by \eqref{defGreenSchr}, where $T$ is an integral operator with positive quasi-metric kernel. 
The case of $C^{1,1}$ domains $\Omega$ and $d \omega =q \, dx$, where $q\in L^1_{loc}(\Omega, dx)$, was treated earlier in \cite{FV1} for 
small $||T||$, and in \cite{FV2} for $||T||<1$.

4. In the special case of Kato class potentials, or more generally, $G$-bounded perturbations $\omega$ for the Schr\"{o}dinger 
operator  $-\triangle -\omega$, it is known that $\mathcal{G}(x,y)\approx G(x, y)$. In this case,  
the gauge $u_1$ exists, and is uniformly bounded,  if and only if $||T||<1$ (see \cite{CZ}, \cite{Han1}, \cite{Pin}).    
 
 5. For the fractional Schr\"{o}dinger operator $(-\triangle)^{\frac{\alpha}{2}} -\omega$, criteria of the existence of the gauge $u_1$  in the case $0<\alpha<2$  were obtained in
  \cite{FV3}. They are quite different from Corollary \ref{cor} 
 and require no extra boundary restrictions on $\Omega$ like  \eqref{martincritsuff-g}, 
 \eqref{martincritnec-g} in the case $\alpha=2$. 

\section{Pointwise estimates for $u_f$}\label{sec2}

Recall that the Martin kernel is defined by \eqref{martin-K}.  Then 
$M(x, z)$ is a H\"older continuous function in $z\in \partial \Omega$ (\cite{Aik}, Theorem 3). 
It is worth mentioning that in uniform domains, harmonic measure may vanish on some surface balls, and so the Radon-Nykodim derivative formula $M(x,z)=\frac{dH^x }{dH^{x_0}}(z)$, which holds for NTA domains, is no longer available as a means to recover \eqref{martin-K} at every point $z \in \partial \Omega$. Instead, it can be determined via \eqref{martin-K},  so that  \eqref{harmmeasiden} still holds (see \cite{Aik}, p. 122). 

In this case, the Martin representation for every nonnegative harmonic function $h$ in $\Omega$ can be expressed  
in the form 
\begin{equation} \label{martin-rep}
h(x)= \int_{\partial \Omega} M(x,z) \, d \mu_h (z), \qquad x \in \Omega,
\end{equation}
where $\mu_h$ is a  finite Borel measure on $\partial \Omega$ uniquely determined by $h$. 

The connection between  Martin's kernel and harmonic measure in a uniform domain is provided by  the equation (see \cite{Aik}, p. 142):  
\begin{equation} \label{hm-mu1}
dH^x (z) = M(x,z) \, d \mu_1 (z), \qquad x \in \Omega, \, z \in \partial \Omega.
\end{equation}
Here $\mu_1$ is the representing measure in  \eqref{martin-rep} for the function $h \equiv 1$. 

Equation  \eqref{hm-mu1} can be justified using  \cite{AG}, Theorem 9.1.7 
(in the special case $h \equiv 1$) for a bounded domain whose Martin boundary $\triangle$ is identified 
with $\partial \Omega$. It yields that, for every $f \in C(\partial \Omega)$, its harmonic extension $Pf$ 
 via harmonic measure  \eqref{harm-rep} can be  represented  in the form 
\begin{equation} \label{pf-rep}
Pf (x) = \int_{\partial \Omega}  M(x,z) \, f(z)  \, d \mu_1 (z), \quad x \in \Omega. 
\end{equation}
By the uniqueness of the representing measure in \eqref{harm-rep} 
for all $f \in C(\partial \Omega)$, it follows that \eqref{hm-mu1} holds. 

In particular, since $M(x_0,z) =1$ for all $z \in \partial \Omega$, 
letting $x=x_0$ in \eqref{hm-mu1} yields $dH^{x_0}= d \mu_1$,
and consequently \eqref{harmmeasiden} holds.

Let $\Omega$ be a bounded uniform domain in $\R^n$, $\omega$ a
finite Borel measure in $\Omega$, and $f\ge 0$ a Borel measurable function in $\partial \Omega$ integrable with respect to harmonic measure.  
We consider solutions $u$ to \eqref{ufeqn} understood in the \textit{potential theoretic}  sense.  
Namely, a function $u: \Omega \rightarrow [0, +\infty]$ is said to be a solution to  \eqref{ufeqn} if 
$u$ is \textit{superharmonic} in $\Omega$ ($u\not\equiv+\infty$),   and  
\begin{equation} \label{u-def}
 u(x) = G(u \omega)(x) + Pf (x), \qquad \text{for all} \, \, x \in \Omega , 
\end{equation}
where $Pf$ is the harmonic function defined by \eqref{harm-rep}. Then $Pf$ is the greatest harmonic minorant of $u$, and 
$u \in L_{loc}^1(\Omega, \omega)$, so that $u \, d \omega$ is the associated Riesz measure of $u$, where  
$-\triangle u = \omega u$ 
in the distributional sense. 
In fact, if 
a potential theoretic solution to \eqref{u-def} exists, then $u \in L^1(\Omega, m \omega)$, where $m(x)=\min (1, G(x, x_0))$ 
for some $x_0 \in \Omega$;  
otherwise $G(u \omega)\equiv +\infty$ (see \cite{AG}, Theorem 4.2.4). 

We note that all potential theoretic solutions are by definition 
lower semicontinuous functions in $\Omega$. For a superharmonic function $u$, it is enough to require that equation \eqref{u-def} holds $dx$-a.e. Moreover, in a bounded uniform domain, any 
potential theoretic solution   $u\in L^1(\Omega, m dx)$. This is not difficult to see using the estimate 
$G (m dx) \le C \, m$ in $\Omega$, which is a consequence of the 
so-called $3$-G inequality (see \cite{CZ}, \cite{Han1}, \cite{Han2}, \cite{Pin}). We remark that in $C^2$ domains $Pf$ is the Poisson integral, and 
in fact $u\in L^1(\Omega, dx)$  (see \cite{FV2}, \cite{MV}, Theorem 1.2.). The latter 
 is no longer true for bounded Lipschitz domains (see, e.g., \cite{MR}). 

Another useful way to define a solution of \eqref{ufeqn}  is to require that  \eqref{u-def} hold
  $d \omega$-a.e.  More precisely, a measurable function $0\le u <+\infty$ $d\omega$-a.e. is said to be a solution 
of \eqref{ufeqn} with respect to $\omega$ if 
\begin{equation} \label{u-omega}
 u = G(u \omega) + Pf  \qquad d \omega\text{-a.e.} \, \, \text{in}\,\, \Omega .
 \end{equation}
If  such a solution exists, then obviously $u \in L^1_{loc} (\Omega, \omega)$, and in fact, as above, $u \in L^1(\Omega, m \omega)$.

We remark that if $f \not= 0$ (with respect to $d H^{x}$), and \eqref{u-omega} has a positive solution in this sense, then $||T||\le 1$ by Schur's lemma, 
and consequently  \eqref{equivnormTless1} holds for $\beta=1$. It follows that 
$\omega (K) \le \text{cap} (K)$ for any compact set $K \subset \Omega$. In particular, 
$\omega$ must be absolutely continuous with respect to the Green (or Wiener) capacity, i.e., 
\begin{equation} \label{abs-cap}
\text{cap} (K)=0 \, \Longrightarrow \, \omega (K)=0.
\end{equation} 
(See details in \cite{FNV}, \cite{FV2}.)

A connection between these two approaches is provided by the following claim used below. If $u$  is a solution of \eqref{u-omega} (with respect to $\omega$), then there exists a unique superharmonic function $\hat u \ge 0$ in $\Omega$ such $\hat u=u$ $d \omega$-a.e. in $\Omega$, and 
 $\hat u \in L^1_{loc} (\Omega, \omega)$ is a potential theoretic solution that satisfies \eqref{u-def}. 

Indeed, let $\hat u:= G(u \omega) + Pf $ everywhere in $\Omega$. Then $\hat u= u$ $d \omega$-a.e.  by \eqref{u-omega}, $\hat u \in L_{loc}^1(\Omega, \omega)$, and consequently 
\[
\hat u(x) =G(u \omega) (x) + Pf (x) = G(\hat u \omega) (x)+ Pf (x) \quad \text{for all} \,\, x \in \Omega .
\]
Clearly, $\hat u$   is superharmonic since $G(u \omega)<+\infty$ $d \omega$-a.e., and hence  $G(u \omega)$ is a Green potential, and 
$Pf$ is the greatest harmonic minorant of $\hat u$. Thus, $\hat u$ is a potential theoretic solution. 

Moreover, such a superharmonic solution $\hat u$ is unique:   if $\hat v$ is a superharmonic solution to \eqref{u-def} for which $\hat v = u$ $d \omega$-a.e., it follows that 
\[
\hat v = G(\hat v \omega) + Pf =G(u \omega) + Pf = \hat u
\]
 everywhere in $\Omega$. 
 
 If $\omega$ satisfies \eqref{abs-cap}, then it is enough 
to require that   $u<+\infty$ and \eqref{u-def} hold q.e. Then $u$ is a solution of \eqref{u-omega} 
with respect to $\omega$, and 
$\hat u:=G(u \omega) + Pf $ is 
a potential theoretic solution to \eqref{ufeqn}, and $\hat u$ is a quasicontinuous representative of $u$, so that 
$\hat u = u$ q.e. 
 
From now on, we will not distinguish between  a solution $u$ to \eqref{u-omega} understood $d \omega$-a.e., and its superharmonic representative $\hat u= u$ $d \omega$-a.e. which satisfies \eqref{u-def} everywhere in $\Omega$. 

In particular, the solution $u_f$ of \eqref{u-def} defined by \eqref{ufsolndef}  
\textit{everywhere} in $\Omega$ is a potential theoretic (superharmonic) solution of 
\eqref{ufeqn} provided $u_f\not\equiv +\infty$. 
Indeed, for $m \in \N$,  
\[
\sum_{j=0}^m T^j(Pf)) (x) = Pf(x) + T \sum_{j=0}^{m-1} T^j(Pf)) (x), \quad \text{for all} \, \, x \in \Omega . 
\]
Letting $m \to \infty$, by the monotone convergence theorem we have  
\begin{align*}
u_f & := \sum_{j=0}^\infty T^j(Pf)\\ & = Pf + T(\sum_{j=0}^\infty T^j(Pf))\\& = Pf + G(u_f \omega)
\end{align*}
 everywhere in $\Omega$. 
 Clearly, $u_f$ is a superharmonic function provided $u_f \not\equiv+\infty$ 
 in $\Omega$, which occurs if and only if $G (u_f \,  \omega)\not\equiv +\infty$ in $\Omega$, or equivalently $u_f \in L^1(\Omega, m \omega)$.  
 Moreover, $u_f$ is a \textit{minimal} solution since, for every other 
 solution $u$, we obviously have, for every  $m \in \N$, 
 \[
 u=G(u \omega) + Pf = G(u \omega) +\sum_{j=0}^m T^j(Pf) \ge \sum_{j=0}^m T^j(Pf).
 \]
Letting $m \rightarrow \infty$, we see that $u \ge u_f$.  



\begin{Def} \label{qmkernel} Let $(\Omega, \omega)$ be a measure space.  A quasi-metric kernel $K$ is a measurable function $K: \Omega \times \Omega \rightarrow (0, +\infty]$ such that $K$ is symmetric ($K(x,y)=K(y,x)$) and $d= \frac{1}{K}$ satisfies 
\[ d(x,y) \leq \kappa (d(x,z) + d(z,y)) \quad \mbox{for all} \quad x,y,z \in \Omega, \]
for some $\kappa >0$, called the \textbf{quasi-metric constant} for $K$.

A measurable function $K: \Omega \times \Omega \rightarrow (0, +\infty]$ is called \textbf{quasi-metrically modifiable} if there exists a measurable function $m: \Omega \rightarrow (0, \infty)$ such that $\tilde{K} (x,y) = \frac{K(x,y)}{m(x)m(y)}$ is a quasi-metric kernel.  The function $m$ is called a \textbf{modifier} for $K$.
\end{Def}

We will use the following result, from \cite{FNV}, Corollary 3.5.   

\begin{Thm} \label{FNVTheorem} Let $(\Omega, \omega)$ be a measure space.  Suppose $K$ is a quasi-metrically modifiable kernel on $\Omega$ with modifier $m$.  Let $\kappa$ be the quasi-metric constant for $\frac{K(x,y)}{m(x)m(y)}$. For a non-negative, measurable function $h$ on $\Omega$, define
\[ Th (x) = \int_{\Omega} K(x,y) h(y) \, d \omega (y), \quad \mbox{for} \,\, x \in \Omega. \]
For $j \in \N$, let $T^j$ be the $j^{th}$ iterate of $T$, and let $T^0 h =h$.  

\vspace{0.1in}

(A) If $\Vert T \Vert <1$, then there exists a positive constant $C$, depending only on $\kappa$ and $\Vert T \Vert$, such that 
\begin{equation} \label{qmkernelupperbnd}
\sum_{j=0}^{\infty} T^j m (x) \leq m(x) e^{C (Tm(x))/m(x)}, \qquad \text{for all} \, \, x \in \Omega .
\end{equation}

\vspace{0.1in}

(B) There exists a positive constant $c$, depending only on $\kappa$, such that 
\begin{equation} \label{qmkernellowerbnd}
\sum_{j=0}^{\infty} T^j m (x)\geq m(x) e^{c (Tm(x))/m(x)}, \qquad \text{for all} \, \, x \in \Omega .
\end{equation}

\end{Thm}

It is known (\cite{An2}, \cite{Han1}) that in a bounded uniform domain $\Omega$ (in particular, an NTA domain), the Green's kernel $G(x,y)$ is quasi-metrically modifiable, with modifier $m(x) = \min (1, G(x,x_0))$, where $x_0$ is any fixed point in $\Omega$, and the quasi-metric constant of 
the modified kernel $G(x,y)/(m (x) \, m(y))$ is independent of $x_0$.

In fact, in a bounded uniform domain $\Omega\subset \R^n$ ($n\ge 3$), the following slightly stronger property (called the strong generalized triangle property) holds (\cite{Han1}, p. 465): 
\begin{equation}\label{strong-quasi} 
 |x_1-x_2| \le |x_1-y| \Longrightarrow \frac{G(x_1, y)}{m(x_1)} \le \kappa \, \frac{G(x_2, y)}{m(x_2)}, 
\end{equation}
for all $x_1, x_2, y \in \Omega$, where $\kappa$ depends only on $\Omega$.  It is known (\cite{Han1}, Corollary 2.8) 
that \eqref{strong-quasi}  is equivalent to the uniform boundary Harnack  principle established for uniform domains in (\cite{Aik}, Theorem 1). By (\ref{strong-quasi}),
\begin{equation}\label{liminf-limsup} 
\limsup_{x_1\rightarrow z, \, x_1 \in \Omega} \frac{G(x_1, y)}{m(x_1)} \le \kappa \, \liminf_{x_2\rightarrow z, \, x_2 \in \Omega}  \frac{G(x_2, y)}{m(x_2)} ,
\end{equation}
for all $y\in\Omega$ and $z \in \partial \Omega$,  where $\kappa$ depends only on $\Omega$, because the condition $|x_1-x_2| \le |x_1-y|$ is satisfied for $x_1$ and $x_2$ sufficiently close to $z$.

We will need the following lemma  for punctured quasi-metric spaces due to Hansen and Netuka (\cite{HN}, Proposition 8.1 and Corollary 8.2); it originated in (Pinchover \cite{Pin}, Lemma A.1) 
for normed spaces. 

\begin{Lemma}\label{hansen} Suppose $d$  is a quasi-metric on a set $\Omega$ with quasi-metric constant $\kappa$. Suppose $x_1 \in \Omega$. Then 
\begin{equation}\label{quasi-cond} 
\tilde d(x,y) = \frac{d(x,y)}{d(x,x_1) \cdot d(y,x_1)}, \qquad x, y 
\in \Omega \setminus\{x_1\},
\end{equation} 
is a quasi-metric on $\Omega\setminus\{x_1\}$ with quasi-metric constant 
$4 \kappa^2$. 
\end{Lemma}

\begin{Lemma} \label{Martinkernelquasimetric} Let $\Omega$ be a bounded uniform  domain with  Green's function $G(x,y)$.  Fix some $x_0 \in \Omega$ and define Martin's kernel $M(x,z)$ for $x \in \Omega$ and $z \in \partial \Omega$ by (\ref{martin-K}).  Then for each $z \in \partial \Omega$, the function $\tilde{m}(x) = M(x,z)$ is a quasi-metric modifier for $G$, with quasi-metric constant $\kappa$ independent of $z \in \partial \Omega$.
\end{Lemma}

\begin{proof} Fix $x_0\in \Omega$, $z\in \partial \Omega$.  As noted above, $m (x) = \min (1, G(x,x_0))$ is a modifier for $G$, so that 
$d(x, y)=\frac{m (x) \, m(y)}{G(x,y)}$ is a quasi-metric on $\Omega$ with positive constant $\kappa$ independent of $x_0$, so that 
\[  \frac{m (x) \, m(y)}{G(x,y)} \leq \kappa \left( \frac{m (x) \, m(w)}{G(x,w)}  + \frac{m (w) \, m(y)}{G(w,y)} \right) ,\]
for all points $x,y,w \in \Omega$.  Suppose $x_1 \in \Omega$ with $x_1 \neq x_0$. Clearly, for $\tilde d$ defined by \eqref{quasi-cond}, we have
\[
\tilde d(x,y) = \frac{1}{m(x_1)^2} \, \frac{G(x, x_1) \, G(y, x_1)}{G(x, y)}, \qquad x, y\in \Omega\setminus\{x_1\} .
\]  
Then by Lemma 
\ref{hansen} it follows that $\tilde d$  is a quasi-metric on $\Omega\setminus\{x_1\}$ with quasi-metric constant 
$4 \kappa^2$. Assuming that $x, y, w \in  \Omega\setminus\{x_1\}$, from  the inequality 
$\tilde d(x,y) \le 4 \kappa^2 [\tilde d(x,w) + \tilde d(y,w)]$, 
we deduce  
\begin{align*}
\frac{1}{m(x_1)^2} \, \frac{G(x, x_1) \, G(y, x_1)}{G(x, y)} &\le \frac{4 \kappa^2}{m(x_1)^2} \, \\
\times & \left[ \frac{G(x, x_1) \, G(w, x_1)}{G(x, w)} + \frac{G(y, x_1) \, G(w, x_1)}{G(y, w)} 
\right]. 
\end{align*}
Multiplying both sides of the preceding inequality by $\frac{m(x_1)^2}{[G(x_0, x_1)]^2}$ yields  
\begin{align*}
 & \frac{G(x, x_1) \, G(y, x_1)}{G(x_0, x_1) \, G(x, y) \, G(x_0, x_1)} \le 4 \kappa^2 \, \\
   & \times \left[ \frac{G(x, x_1) \, G(w, x_1)}{G(x_0, x_1) \, G(x, w) \, G(x_0, x_1)} + \frac{G(y, x_1) \, G(w, x_1)}{G(x_0, x_1) \, G(y, w) \, G(x_0, x_1)} 
\right]. 
\end{align*}

Letting $x_1 \rightarrow z$, with $x_1 \in \Omega$, we have 
\[ 
\lim_{x_1 \rightarrow z, \, x_1 \in \Omega} \frac{G(x, x_1)}{G(x_0, x_1)}= M(x,z)=\tilde{m} (x), \]
by (\ref{martin-K}), and similarly with $x$ replaced by $y$ or $w$.  We obtain 
\[  \frac{\tilde{m} (x)\tilde{m}(y)}{G(x,y)} \leq 4 \kappa^2 \left( \frac{\tilde{m} (x)\tilde{m}(w)}{G(x,w)}  + \frac{\tilde{m}(w)\tilde{m}(y)}{G(w,y)} \right) . \]
\end{proof}

\bigskip

\begin{proofof} Theorem \ref{mainufest}.  By Lemma \ref{Martinkernelquasimetric}, $\tilde{m}(x)=M(x,z)$ is a quasi-metric modifier for $T$, for all $z \in \partial \Omega$, with quasi-metric constant independent of $z$.  Hence by part (A) of Theorem \ref{FNVTheorem} with $\tilde m$ in place of $m$, under the assumption that $\Vert T \Vert <1$, 
(note that the estimates in Theorem \ref{FNVTheorem} hold everywhere)
\begin{equation}\label{upperTM}
\begin{aligned}
 \mathcal{M}(x, z) & = \sum_{j=0}^{\infty}  T^j M(\cdot, z) (x)   \leq M(x,z) e^{C \, (TM(\cdot,z))(x)/M(x, z)}\\ & = M(x,z) e^{C \int_{\Omega} G(x,y) \frac{M(y,z)}{M(x,z)} d \omega(y)}  , \qquad (x, z) \in \Omega\times\partial \Omega , 
 \end{aligned}
 \end{equation} 
with $C$ depending only on $\Omega$ and $||T||$.  Substituting this estimate in \eqref{ufrepresent} and using equation (\ref{harmmeasiden}) gives (\ref{ptwiseupperbnd}). This proves part (A) of Theorem \ref{mainufest}. 

Suppose now that $u$ is a solution to \eqref{ufeqn}. Assuming without loss of generality that $f \not=0$ $d H^{x}$-a.e., so that $u \geq Pf>0$ 
is a positive solution,  we see that $T u \leq u$, where $0<u<\infty$ $d \omega$-a.e. Hence, $\Vert T \Vert \leq 1$, 
and  consequently \eqref{equivnormTless1} holds with $\beta=1$,  
by Schur's lemma (see \cite{FNV}, \cite{FV2}). In particular, \eqref{abs-cap} holds.

Since $Pf$ is a positive harmonic function, obviously $Pf\ge c_K>0$ on every compact set $K\subset \Omega$, 
 and consequently 
\begin{equation}\label{F-M-dom}
 c_K \, G(\chi_K \omega)\le G(Pf \omega) \le G (u \omega) \le u < \infty \quad d \omega\text{-a.e.} 
 \end{equation}
 This simple observation will be used below. 

For the minimal solution $u_f$  to \eqref{ufeqn}   
given by \eqref{ufsolndef} we have  $u\ge u_f$. 
Applying part (B) of Theorem \ref{FNVTheorem} with $\tilde m=M(\cdot, z)$ in place of $m$ gives 
\begin{equation}\label{lowerTM}
\begin{aligned}
\mathcal{M}(x, z) & =  \sum_{j=0}^{\infty}  T^j M(\cdot, z) (x)    \geq M(x,z) e^{c \, (TM(\cdot,z))(x)/M(x, z)}\\ & = M(x,z) e^{c \int_{\Omega} G(x,y) \frac{M(y,z)}{M(x,z)} d \omega(y)}  , 
\qquad (x, z) \in \Omega\times\partial \Omega ,
 \end{aligned}
 \end{equation} 
with $c$ depending only on $\Omega$. 

In fact, we can let $c=1$ in \eqref{lowerTM} if instead of statement (B) of Theorem \ref{FNVTheorem}  we use 
a recent lower estimate of solutions obtained in \cite{GV2}, Theorem 1.2, with $q=1$, 
$\mathfrak{b}=1$, and $h=\tilde m$. 
Here $\mathfrak{b}$ is the constant in the so-called weak domination principle, which states that, 
for any bounded measurable function $g$ with compact support,  
\begin{equation}\label{weak-dom} 
G (g \omega)(x)\leq h(x)\ \text{in }\mathrm{supp} (g)\ \ \Longrightarrow \ \
G(g \omega)(x) \leq \mathfrak{b}\ h(x) \, \, \text{in\ }\Omega ,
\end{equation}
where $h$ is a given positive lower semicontinuous function on $\Omega$. 

For Green's kernel $G$, this property with $\mathfrak{b}=1$ is a consequence of the classical Maria--Frostman 
 domination principle (see \cite{Hel}, Theorem 5.4.8), for any  positive superharmonic function $h$.  We only need to verify that 
 $G (g \omega)<\infty$ $d \omega$-a.e., which is immediate from  \eqref{F-M-dom}. Hence, 
 \eqref{weak-dom} holds with $\mathfrak{b}=1$, and so 
 \eqref{lowerTM} holds with $c=1$ by  \cite{GV2}, Theorem 1.2.

Consequently,  by the same argument as above, 
 \begin{align*}
 u_f (x)& = \int_{\partial \Omega} f(z) \, 
   \sum_{j=0}^{\infty} T^j M(\cdot, z) (x)  \, d H^{x_0}(z)\\ & \geq 
  \int_{\partial \Omega} e^{\int_{\Omega} G(x,y) \frac{M(y,z)}{M(x,z)} d \omega(y)}\,  f(z) \, M(x,z)   \, d H^{x_0}(z), \quad \text{for all} \, \, x \in \Omega,
 \end{align*}
where $M(x,z) \, H^{x_0}(z)= d H^{x}(z)$. This yields the lower bound (\ref{ptwiselowerbnd}), 
  The proof of part (B) of Theorem \ref{mainufest} is complete. 
\end{proofof}

We complete this section with an extension of Theorem \ref{mainufest} which covers 
solutions of \eqref{form-sol} with an arbitrary positive harmonic function $h$ in place of $Pf$.  Such solutions arise naturally, because, if $u$ positive superharmonic function in $\Omega$ such that 
\begin{equation}\label{u-harm} 
-\triangle  u = \omega u, \quad u\ge 0, \, \,  \mbox{in} \, \,  \Omega , 
\end{equation} 
and if the greatest harmonic minorant of $u$ is $h>0$, then by the Riesz decomposition theorem, 
\begin{equation}\label{u-harm-int}
u= G(u \omega)+ h \quad u\ge 0, \, \,  \mbox{in} \, \,  \Omega ,
\end{equation} 
where $G(u \omega)\not\equiv +\infty$, and $u d \omega$ is the corresponding Riesz measure, a locally finite Borel measure in $\Omega$.  

Given a positive harmonic function $h$ on $\Omega$, we will estimate the minimal solution 
\[ u_h = h + \sum_{j=1}^\infty T^j h \]
of \eqref{u-harm-int} and in particular give conditions for $u_h$ to exist, i.e., such that $u_h \not\equiv +\infty$.  The proof is based on Martin's representation \eqref{martin-rep}, which takes the place of \eqref{harm-rep} in the proof of Theorem \ref{mainufest}.

\begin{Thm} \label{mainu-harm}  Let $\Omega\subset \R^n$ be a bounded uniform domain, 
 $\omega$ a locally finite Borel measure on $\Omega$, and $h$ a positive harmonic 
  function in $\Omega$.  

\vspace{0.1in}

(A)  If $\Vert T \Vert <1$, then there exists a positive constant $C$ depending only on $\Omega$ and $\Vert T \Vert$ such that 
\begin{equation} \label{upperbnd-harm}
u_h (x) \leq  \int_{\partial \Omega} e^{C \int_{\Omega} G(x,y) \frac{M(y,z)}{M(x,z)} d \omega (y)} M(x, z) \, d \mu_h(z), \quad x \in \Omega .
\end{equation}

\vspace{0.1in}

(B) If $u$ is a positive  solution of \eqref{u-harm-int}, then $\Vert T \Vert \leq 1$, and 
\begin{equation} \label{lowerbnd-harm}
  u (x) \geq  \int_{\partial \Omega} e^{\int_{\Omega} G(x,y) \frac{M(y,z)}{M(x,z)} d \omega (y)} M(x, z) \, 
  d \mu_h(z), \quad x \in \Omega . 
\end{equation}
\end{Thm}

The proof of Theorem \ref{mainu-harm} is very similar to that of Theorem \ref{mainufest} 
above. We only need to integrate both sides of estimates \eqref{upperTM} and \eqref{lowerTM} 
over $\partial\Omega$ against $d \mu_h(z)$ in place of  $f(z) \, dH^{x_0} (z)$.

\section{Existence criteria for $u_f$}\label{sec3}

We require a few results prior to giving the proof of Theorem \ref{gaugecrit}. The following lemma is well-known (see, for instance,  \cite{AG}, Lemma 4.1.8 and Theorem 5.7.4), but we include a proof for the sake of completeness.  Recall that $x_0 \in \Omega$ is a fixed reference point and $m(x) = \min (1, G(x,x_0))$.

\begin{Lemma} \label{estGchiK} Let $\Omega \subseteq{\R^n}$ ($n\ge 2$) be a domain 
with nontrivial Green's function $G$.  Let $K$ be a compact subset of $\Omega$ and let $\chi_K$ be the characteristic function of $K$. There exists a constant $C_K$ depending on $\Omega$, $K$, and the choice of $x_0$, such that
\begin{equation}\label{GchiKest}
  G \chi_K (x) \leq C_K \,  m (x), \quad x \in \Omega. 
\end{equation}

Also, if $|K|>0$, there exists a constant $c_K>0$ depending on  $\Omega$, $K$ and $x_0$ such that
\begin{equation}\label{lowerGchiKest}
  G \chi_K (x) \geq c_K  \, m (x), \quad x \in \Omega. 
\end{equation}
\end{Lemma}

\begin{proof} We first prove inequality (\ref{GchiKest}).  Suppose $n \geq 3$  (the case $n=2$ is handled in a similar way with obvious modifications). We assume $|K|>0$, else the result is trivial.  We also assume that $x_0 \in K$; if not, replacing $K$ with $K \cup \{x_0\}$ does not change $G \chi_K$.  We first claim that there exists a constant $C_1(K)$ depending on $K$ and $x_0$ such that 
\begin{equation} \label{GchiKbnd1}
G \chi_K (x) \leq C_1 (K) ,
\end{equation}
for all $x \in \Omega$.  To prove this claim, we recall the standard fact that $G(x,y) \leq C|x-y|^{2-n}$ for all $x, y \in \Omega$.  Let $R$ be the diameter of $K$.  Then there exists $y_0 \in K$ such that $K \subseteq \overline{B(y_0, R)}$.  If $x \in B(y_0, 2R)$, then $K \subseteq B(x, 3R)$ and 
\[  \int_K G(x,y) \, dy \leq \int_{B(x, 3R)} \frac{c}{|x-y|^{n-2}} \, dy \leq c \int_0^{3R} \frac{r^{n-1}}{r^{n-2}} \, dr = c R^2.   \] 
If $x \not\in B(y_0, 2R)$ then $|x-y|^{2-n} \leq R^{2-n}$ for all $y \in K$, so
\[  \int_K G(x,y) \, dy \leq C R^{2-n} |K| \leq c R^2 . \]

Next we claim that there exists a constant $C_2$ depending on $\Omega$, $K$ and $x_0$ such that 
\begin{equation} \label{GchiKbnd2}
G \chi_K (x) \leq C_2  \, G(x, x_0) ,
\end{equation}
for all $x \in \Omega$.  For this claim, let $U$ be a subdomain of $\Omega$ such that 
$x_0\in U$, $K \subseteq U$ and $\overline{U} \subseteq \Omega$.  
If $x \in \Omega \setminus U$, then $G(x,y)$ is a positive harmonic function of $y$ in $U$, so by Harnack's inequality (e.g., see \cite{AG}, Corollary 1.4.4), there exists a constant  $C(K, U)$ such that $G(x,y) \leq C(K, U) \, G(x, x_0)$ for all $y \in K$.  Hence 
\[ \int_K G(x,y) \, dy \leq  C(K, U) \, |K| \, G(x, x_0).\]
 Since  a fixed domain $U$ depends only on $x_0$, $K$, and $\Omega$, we can replace $C(K, U)$ with $C(x_0, K, \Omega)$.  
On the other hand, suppose $x \in U$. Note that $G(z, x_0)$ is a strictly positive lower semi-continuous function of $z \in \Omega$ and hence $M = \min \{ G(z, x_0): \, z \in \overline{U} \} >0$, where $M$ depends on $\Omega, x_0$ and $U$, hence $K$. Hence by equation (\ref{GchiKbnd1}),
\[  G \chi_K (x) \leq C_1 (K) \leq \frac{C_1(K)}{M} \, G(x, x_0). \]

Since $m(x) = \min (1, G(x, x_0))$, inequalities (\ref{GchiKbnd1}) and (\ref{GchiKbnd2}) imply inequality (\ref{GchiKest}).

To prove inequality (\ref{lowerGchiKest}), let $U$ be as above. For $x \in \Omega \setminus U$, the same application of Harnack's inequality as above gives that $G(x,y) \geq C (x_0, K, \Omega)^{-1} G(x, x_0)$ for all $y \in K$.  Hence 
\[   \int_K G(x,y) \, dy \geq  C(K, \Omega)^{-1} \, |K| \, G(x, x_0) \geq C(K, \Omega)^{-1} |K| \,  m(x). \]
Now suppose $x \in \overline{U}$.  Note that $G(z,y) $ is a strictly positive lower semi-continuous function of $(z,y)$ in $\Omega\times \Omega$ (see \cite{AG}, Theorem 4.1.9).  Hence $C_3 (\overline{U}) = \min \{ G(z,y) \,: \, 
(z, y) \in \overline{U}\times  \overline{U}$ is attained at some point in the compact set $\overline{U}\times  \overline{U}$. 
 In particular, $C_3(\overline{U})>0$.  Since $m(x) \leq 1$, 
\[ \int_K G(x,y) \, dy \geq C_3 (\overline{U}) \, |K| = C_3 (x_0, K, \Omega) \, m(x) . \]
\end{proof}

\begin{Lemma}\label{low-M-est} Suppose $\Omega \subset \R^n$ ($n\ge 2$)  is a bounded uniform domain.  
Suppose $x_0\in \Omega$ is a reference point for the  
Martin kernel. Then there exists a positive constant $c$ depending only on $x_0$ and 
$\Omega$ such that 
\begin{equation}\label{martin-low}
M(x, z) \ge c \, m(x), \qquad \text{for all} \, \,  (x, z) \in \Omega\times\partial \Omega ,
\end{equation}
where $m(x)=\min (1, G(x, x_0))$.

 In particular, if $\omega$ is a locally finite Borel measure in $\Omega$ such that $M^{*} (m \, \omega) \not\equiv+\infty$, then $m \in L^2 (\Omega, \omega)$. 
\end{Lemma}

\begin{proof} Fix $z \in \partial \Omega$. Let  $B(x_0, r) \subset \Omega$, 
where $0<r\le \frac{1}{2} \, \text{dist} \, (x_0, \partial \Omega)$. Since $M(\cdot, z)$ is a positive  harmonic 
function in $\Omega$,  by Harnack's inequality in $B(x_0, 2r)$,   
there exists a constant $c>0$ depending only on $x_0$ and $r$ such that $M(x, z) \ge c \, M(x_0, z)$, 
for all $x \in B(x_0, r)$ where $M(x_0, z)=1$. Hence,  
\begin{equation}\label{m-c}
M(x, z) \geq c >0, \quad \text{for all} \, \, x \in B(x_0, r) . 
\end{equation} 

For $x \in \Omega\setminus B(x_0, r)$, we argue that by the $3$-G inequality in a 
bounded uniform domain ($n\ge 3$), 
\[
\frac{G(x, x_0) \, G(x_0, y)}{G(x, y)}\le C \, \left( |x-x_0|^{2-n} + |y-x_0|^{2-n} \right) , 
\] 
for all $y \in \Omega$, where $C$ depends only on $\Omega$, see \cite{Han1}. Hence, for $x, y \in  \Omega\setminus B(x_0, r)$, 
\[
\frac{G(x, y)}{G(x_0, y)}\ge C^{-1} \, \frac{G(x, x_0)}{|x-x_0|^{2-n} + |y-x_0|^{2-n} }\ge  C^{-1} 2 r^{n-2} \, G(x, x_0). 
\]
(For $n=2$, an analogue of the $3$-G inequality holds in any bounded domain \cite{Han2}.) 
Letting $y \rightarrow z$,  where without loss of generality we may assume that $y \in  \Omega\setminus B(x_0, r)$, we deduce 
\begin{equation}\label{m-g}
M(x, z) \ge   C^{-1} 2 r^{n-2}\, G(x, x_0) , \quad \text{for all} \, \, x \in  \Omega\setminus B(x_0, r) . 
\end{equation} 
Combining  estimates \eqref{m-c} and \eqref{m-g} 
yields \eqref{martin-low}. 

If $\omega$ is a locally finite Borel measure in $\Omega$ such that $M^{*} (m \, \omega)\not\equiv
+\infty$, then  for some $z\in \partial \Omega$ by \eqref{martin-low}  $\int_\Omega m^2 d \omega
\le c M^* (m \omega)(z) <+\infty$, i.e., $m \in L^2(\Omega, \omega)$.
\end{proof}

\bigskip

\begin{Lemma}\label{conv-lemma}
Suppose $\Omega \subset \R^n$ ($n\ge 2$)  is a bounded uniform domain.   
Suppose $\mu$ is a finite Borel measure  with compact support in $\Omega$.    Let $z \in \partial \Omega$. 
Then 
\begin{equation}\label{min-thin2}
 \lim_{x \rightarrow z, \, x\in\Omega} \frac{G \mu (x)}{G(x, x_0)} = \int_{\Omega} M(y,z) \,  d\mu(y) = M^* \mu (z). 
 \end{equation}
In addition, if $z$ is a regular point of $\partial \Omega$, then 
\begin{equation}\label{min-thin}
 \lim_{x \rightarrow z, \, x\in\Omega} \frac{G \mu (x)}{m(x)} = \int_{\Omega} M(y,z) \,  d\mu(y) = M^* \mu (z).
 \end{equation}
 \end{Lemma}

\begin{proof}   By \eqref{martin-K}, if $y\in\Omega$ and $x_j\rightarrow z$ ($x_j\in \Omega$),  then 
\[
\lim_{j \rightarrow \infty} G(y, x_j) / G (x_j, x_0) = M(y,z) . 
\]
As in the proof of Lemma \ref{estGchiK}, we denote by $U$ a relatively compact domain in $\Omega$ 
that contains both $x_0$ and $K$. Since $x_j \rightarrow z$, where $z\in \partial\Omega$, we have that $x_j \not \in \overline{U}$ for $j \ge j_0$. Then  $G(y, x_j)$ is a harmonic function of $y\in U$, and for each $j \ge j_0$,  by Harnack's inequality, 
\[
G(y, x_j) \le C(K, U) \, G(x_0, x_j), \qquad \text{for all} \, \, y \in K . 
\]
 Since $\mu$ is a finite measure, we obtain \eqref{min-thin2} by the dominated 
convergence theorem. 

If $z$ is a regular point of $\partial \Omega$, then $G (x_j, x_0)\rightarrow 0$ as $j \rightarrow \infty$, 
and consequently $m(x_j)=G (x_j, x_0)$ for $j$ large enough. Hence,  \eqref{min-thin} follows from 
 \eqref{min-thin2}.

\end{proof} 

In Lemma \ref{conv-lemma}, $\mu$ is a finite Borel measure  with compact support in $\Omega$.  We remark that more generally, for $\mu$ only locally finite,  
\begin{equation}\label{martin1}
\liminf_{x\rightarrow z, \, \, x\in \Omega} \frac{G \mu(x)}{G(x_0, x)}\ge \int_\Omega M(x, z) \, d \mu(x),
\end{equation}
for $z\in \triangle$ (a Martin boundary point), by Fatou's Lemma.  In fact, by \cite{AG}, Theorem 9.2.7, for any Green's potential $G \mu$ and $z\in \triangle_1$  
(a Martin boundary point where $\Omega$ is not minimally thin), 
\begin{equation}\label{martin2}
\liminf_{x\rightarrow z, \, \, x\in \Omega} \frac{G \mu(x)}{G(x_0, x)}= \int_\Omega M(x, z) \, d \mu(x).
\end{equation}
For uniform domains, $\triangle=\triangle_1=\partial\Omega$, so that \eqref{martin2} holds for all 
$z\in \partial\Omega$. We could use this fact in our proof below, but we prefer the more elementary approach in Lemma \ref{conv-lemma}. The compact support restriction can be removed later in the proof by exhausting $\Omega$ with a sequence of nested domains $\Omega_j$, and using the monotone convergence theorem.

\vspace{0.1in}

\begin{proofof} Theorem \ref{gaugecrit}. 

\vspace{0.1in}

(A) Suppose $\Vert T \Vert <1$ and \eqref{martincritsuff} holds. Define 
\begin{equation}\label{defngreenpot}
  Gf (x) = \int_{\Omega} G(x,y) f(y) \, dy,   \,\,\, x \in \Omega. 
\end{equation}
Let $G_1 =G$, and let  $G_j (x,y)$ be the kernel of the $j^{th}$ iterate $T^j$ of $T$ defined by (\ref{defT}), so that 
\begin{equation} \label{defGj}
T^j h(x) = \int_{\Omega} G_j(x,y) h(y)\, d \omega (y).
\end{equation}
Then  $G_j$ in \eqref{defGj} is determined inductively for $j \geq 2$ by
\begin{equation}\label{new-defGj}
 G_j (x,y) = \int_{\Omega} G_{j-1} (x, w) G(w,y) \,  d\omega(w). 
\end{equation}
We define the minimal Green's function associated with the Schr\"{o}dinger operator $-\triangle - \omega$ to be   
\begin{equation}\label{defGreenSchr}
 \mathcal{G} (x,y) = \sum_{j=1}^{\infty} G_j (x,y),  \qquad \text{for all} \, \,  x, y \in \Omega.  
 \end{equation} 
 The corresponding Green's operator is 
\[  \mathcal{G}f(x)  = \int_{\Omega} \mathcal{G}(x,y) f(y) \, dy, \quad x \in \Omega .\] 

Let 
$K$ be a compact set in $\Omega$. Denote by $u_K$ a solution to the equation
\begin{equation}\label{K-eqn} 
\left\{ \begin{aligned}
-\triangle  u & = \omega u + \chi_K\, \,&  &  \mbox{in} \, \,  \Omega, \quad  u \ge 0,  \\
u & = 0 \, \, &  &\mbox{on} \, \,  \partial \Omega .
\end{aligned}
\right.  
\end{equation} 
 In other words, 
\begin{equation} \label{eqnforuK}
u_K = G(u_K \omega) + G\chi_K . 
\end{equation}

By Lemma \ref{estGchiK},  $G\chi_K (x)\approx m(x)$ in $\Omega$ if   
 $m(x)=\min(1, G(x, x_0))$. 
 Without loss of generality we may assume that $m \in L^2 (\Omega, \omega)$; otherwise 
 $M^*(m \, \omega) \equiv +\infty$ by Lemma \ref{low-M-est}, and 
 condition  \eqref{martincritsuff} is not valid. It follows that $G \chi_K \in  L^2 (\Omega, \omega)$. 
 But $||T||<1$, so that $u_K=(I-T)^{-1} G \chi_K\in L^2(\Omega, \omega)$, and the series in \eqref{defuK} converges in 
 $L^2(\Omega, \omega)$ (and hence $d \omega$-a.e.).  In particular, $G(u_K \omega)\not\equiv \infty$.

From this fact it is immediate that the minimal superharmonic solution to \eqref{eqnforuK} is given by 
\begin{equation} \label{defuK}
\begin{aligned} 
u_K (x) & := G(u_K \omega) + G\chi_K = (I-T)^{-1} G \chi_K (x) \\& = \sum_{j=0}^{\infty} T^j (G \chi_K) (x)= \int_{K} \mathcal{G} (x,y) \, dy ,
\end{aligned} 
\end{equation}
for all $x \in \Omega$. 

By equation (\ref{ufsolndef}), 
\begin{align*}  
 u_f (x)  & = Pf(x) + \sum_{j=1}^{\infty} T^j (Pf)(x) \\
  & =  Pf(x)+   \int_{\Omega}  \mathcal{G} (x,y) \, Pf (y) \, d\omega(y),
 \end{align*}
for all $x \in \Omega$. Integrating both sides of this equation over $K$ with respect to $dx$, 
\begin{equation}\label{intKuf}
\begin{aligned}  
\int_{K} u_f (x) \, dx & = \int_K Pf(x) \, dx + \int_{K} \int_{\Omega} \mathcal{G} (x,y) \, Pf(y) \, d \omega (y) \, dx \\
  & = \int_K Pf(x) \, dx +  \int_{\Omega}  \int_{K}  \mathcal{G} (x,y) \, dx \, Pf (y) \, d\omega(y) \\ & =   \int_K Pf(x) \, dx  +  \int_{\Omega} u_K (y)\, Pf (y) \, d\omega(y) ,
\end{aligned}
\end{equation}
by Fubini's theorem, equation (\ref{defuK}) and the symmetry of $\mathcal{G}$. 

The term $\int_K Pf(x) \, dx $ is finite because \eqref{martincritsuff} guarantees that $f$ is integrable with respect to harmonic measure, so $Pf$ is not identically infinite, and so is harmonic.  Thus to prove that $u_f \in L^1 (K, dx)$, it suffices to show that $u_K Pf \in L^1 (\Omega, \omega)$.  

By \eqref{pf-rep-martin} and Fubini's theorem, 
\[
\int_{\Omega} u_K(y) \, Pf(y) \, d\omega(y)   = \int_{\partial \Omega}  \int_{\Omega} M(y,z) u_K (y) \, d \omega (y) \, f(z) \, dH^{x_0} (z) . 
\]

 We claim that 
\begin{equation} \label{Mu0est}
\int_{\Omega} M(y,z) u_K (y) \, d \omega (y) \leq C_K \, e^{C \,  M^* (m\omega) (z)}, 
\end{equation}
if $z$ is a regular point of $\partial\Omega$. Assuming \eqref{Mu0est} for the moment, 
the set of irregular boundary points $E \subset \partial\Omega$ 
 is known to be Borel and polar, i.e., $\text{cap}(E)=0$ (\cite{AG}, Theorem 6.6.8), and consequently 
 negligible, i.e., of harmonic measure zero (\cite{AG}, Theorem 6.5.5). Therefore \eqref{Mu0est} yields
\begin{equation} \label{u_K-Pf}
 \int_{\Omega} u_K(y) \, Pf(y) \, d\omega(y) 
 \leq C_K \int_{\partial \Omega}  e^{C \, M^* (m\omega) (z)} \, 
f(z) \, d H^{x_0} (z).
 \end{equation}
Hence our assumption \eqref{martincritsuff} guarantees that $u_K \, Pf  \in L^1 (\Omega, \omega)$.

To prove (\ref{Mu0est}), let us assume first that $\omega$ is compactly supported. Then as mentioned above after \eqref{eqnforuK}, $u_K \in L^2 (\Omega,  \omega)$. Hence, by Cauchy's inequality, $d \mu= u_K \, d \omega$ is a finite compactly supported measure.
By equation \eqref{defuK}, Lemma \ref{estGchiK}, and Theorem \ref{FNVTheorem}, 
\begin{equation} \label{uKexpest}
  u_K (x) \leq  C_K \sum_{j=0}^{\infty} T^j m (x) \leq C_K \,  m(x) \, e^{C G(m\omega)(x) /m(x) },
	\end{equation}
since $Tm= G(m \omega)$. Using the trivial estimate $m(\cdot) \le G(x_0, \cdot)$, followed by \eqref{eqnforuK} and then \eqref{uKexpest},
\begin{equation} \label{GuKomegaest}
 \frac{G(u_K\omega)(x)}{G(x, x_0)} \leq \frac{G(u_K\omega)(x)}{m(x)} \leq  \frac{u_K(x)}{m(x)} \leq C_K e^{C G(m\omega)(x)/m(x)} ,
\end{equation}
for $x \in \Omega$.  Applying \eqref{min-thin2} with $d\mu = u_K d\omega$ and then \eqref{min-thin} with $d\mu = m d\omega$,
\begin{align*}
  \int_{\Omega} M(y,z) u_K (y) \, d \omega (y) & = \lim_{x \rightarrow z, x \in \Omega} \frac{G(u_K\omega)(x)}{G(x, x_0)} \\
	& \leq \lim_{x \rightarrow z, x \in \Omega} C_K e^{C G(m\omega)(x)/m(x)} \\
	& = C_K e^{ M^* (m\omega) (z) } , 
	\end{align*}
where the regularity of $z \in \partial \Omega$ is used only at the last step.  Hence \eqref{Mu0est} is established for compactly supported measures $\omega$.

In the general case, consider an exhaustion  $\Omega=\cup_{k=1}^{\infty} \Omega_k$, where $\{\Omega_k\}$ is a family of nested,  relatively compact subdomains of $\Omega$. Without loss of generality we may assume that $x_0\in \Omega_k$, for all $k \in \N$. 

   In $\Omega\times\Omega$, define the iterated Green's kernels $G_j^{(k)} (x,y)$ for $j\in \N$, and $\mathcal{G}^{(k)}(x,y)=\sum_{j=1}^\infty G_j^{(k)} (x,y)$, as in  (\ref{new-defGj}), (\ref{defGreenSchr}), except with $\omega$ replaced by $\omega_k$, $k \in \N$.  Let $u_K^{(k)} = \mathcal{G}^{(k)} \chi_K$.  By repeated use of the monotone convergence theorem, we see that $G_j^{(k)} (x,y)$  increases monotonically as $k \rightarrow \infty$ to $G_j (x,y)$ for each $j$, $\mathcal{G}^{(k)}(x,y)$ increases monotonically to $\mathcal{G} (x,y)$, and $u_K^{(k)}$ increases monotonically to $u_K$.  Applying the compact support case gives  
\begin{align*} 
\int_{\Omega} M(y,z) u_K^{(k)} (y) \ \chi_{\Omega_k}(y) \, d\omega(y) & \leq C_K e^{C \,   M^* (m \, \omega_k) (z)} \\ & \leq   C_K e^{C \,  M^* (m \, \omega) (z)} . 
\end{align*} 
Then, as $k \rightarrow \infty$, the monotone convergence theorem yields (\ref{Mu0est}).

\vspace{0.1in}

(B)  Suppose $u_f \in L^1_{loc} (\Omega, dx)$, where $f \not= 0$ a.e. relative to harmonic measure, and 
\[ u_f = Tu_f + Pf  \qquad  \,\,  \text{on} \, \,   \Omega .\]
So $Tu_f \leq u_f$, where $0<u_f<\infty$  $d \omega$-a.e. 
It follows  by Schur's lemma that $\Vert T \Vert_{L^2 (\omega)\rightarrow L^2 (\omega)} \leq 1$. 

It remains to show that \eqref{martincritnec}  holds. 
We remark that this condition   follows immediately from  \eqref{ptwiselowerbnd} with 
 $x=x_0$ provided $u_f(x_0) < \infty$. Since this is not necessarily the case, we proceed as follows.

Choose any compact set $K \subseteq \Omega$ with $|K|>0$. 
By Lemma \ref{estGchiK} and Theorem \ref{FNVTheorem}, 
\begin{equation} \label{lowestuK}
 u_K (x) = \sum_{j=0}^{\infty} T^j G\chi_K (x) \geq c_K \sum_{j=0}^{\infty} T^j m (x)  \geq c_K m(x) e^{c(Tm(x))/(m(x))}, 
\end{equation}
for all $x \in \Omega$. In fact, we can let $c=1$ in the preceding 
 estimate, exactly as in the proof of  \eqref{lowerTM} above, by using \cite{GV2}, Theorem 1.2 with $q=1$, $h=m$, and $\mathfrak{b}=1$.   Notice 
 that $m$ is a superharmonic function in $\Omega$, and so the Maria-Frostman domination 
 principle yields \eqref{weak-dom} with $\mathfrak{b}=1$ and $h=m$. 

By inequality \eqref{lowestuK}, equation (\ref{eqnforuK}) and inequality (\ref{GchiKest}), 
\begin{equation} \label{Tm-m}
\begin{aligned}
 e^{Tm(x)/(m(x))} & \leq c_K^{-1} \frac{u_K(x)}{m(x)} = c_K^{-1} \left( \frac{G(u_K \omega)(x)}{m(x)} + \frac{G\chi_K (x)}{m(x)} \right) 
\\ & \leq c_K^{-1} \frac{G(u_K \omega)(x)}{m(x)} + C_K c_K^{-1} . 
 \end{aligned}
 \end{equation}
 
 Let $z \in \partial \Omega$ be a regular point. Applying Lemma \ref{conv-lemma} with $d\mu= m d\omega$ on the left side of (\ref{Tm-m}) (recalling that $Tm= G(m \omega)$), and with $d\mu = u_K \omega$ on the right side, we obtain 
\begin{equation}  \label{M-claim}
 e^{ M^* (m \omega) (z)}  \leq  c_K^{-1}  \, \int_{\Omega} M(y,z) u_K (y) \, d \omega (y) + C_K c_K^{-1},  
\end{equation}
if $\omega$ has compact support in $\Omega$. By the same exhaustion process that was used in the opposite direction, (\ref{M-claim}) holds for $\omega$ locally finite in $\Omega$.

Since the set of irregular points in $\partial \Omega$ has harmonic measure $0$, as noted above, we can integrate \eqref{M-claim} over $\partial \Omega$ with respect to $f \, dH^{x_0}$  and apply Fubini's theorem to obtain 
\begin{align*}
 &  \int_{\partial \Omega}   e^{M^* (m \omega) (z)} \, f(z) \, dH^{x_0} (z) \\ &
 \leq C_1 c_K^{-1} \int_{\Omega} \int_{\partial \Omega} M(y,z) \, f(z) \, dH^{x_0} (z) u_K (y) \, d \omega (y) \\&  + C_K c_K^{-1}  \int_{\partial \Omega} \, f(z) \, dH^{x_0} (z)  \\
& = C_1 c_K^{-1}  \int_{\Omega} u_K (y) \, Pf(y) \, d \omega (y)  + C_K c_K^{-1}  \int_{\partial \Omega} \, f(z) \, dH^{x_0} (z) ,   
\end{align*}
using equation \eqref{pf-rep-martin}.  Since $u_K \, Pf \in L^1 (\Omega, \omega)$ by \eqref{intKuf}, we have condition \eqref{martincritnec}.      
  \end{proofof}

\noindent{\bf Remark.}   For part (A) of Theorem \ref{gaugecrit}  and Corollary \ref{cor}, if $\Omega$ is a bounded 
$C^{1,1}$ domain, or a bounded Lipschitz domain with sufficiently small Lipschitz constant, then $G\chi_{\Omega} \approx m$ (see, for instance, \cite{AAC}, Theorem 1.1 and Remark 1.2(i)). 
Hence, $\int_\Omega M(x, z) \, dx\le C$, where $C$ does not depend on $z \in\partial \Omega$. Then one can replace $\chi_K$ above with $\chi_{\Omega}$ and obtain that $u_f \in L^1 (\Omega, dx)$ with 
\[   \int_{\Omega} u_f (x) \, dx \leq  C \int_{\partial \Omega}  f(z) \, dH^{x_0} (z)
 +  C \int_{\partial \Omega}  e^{CM^* (m \omega) (z)} \, f(z) \, dH^{x_0} (z) . \]

In the same way that Theorem \ref{mainu-harm} generalizes Theorem \ref{mainufest}, there is a complete analogue of Theorem \ref{gaugecrit} for solutions of  equation \eqref{u-harm-int}, with an arbitrary positive  harmonic function $h$ in place of $Pf$. It gives sufficient and matching necessary conditions for the existence of solutions whose pointwise estimates are provided in Theorem \ref{mainu-harm}.  The primary difference in this case is that  $\mu_h$ is not necessarily zero on the set of irregular points of $\partial\Omega$.   Hence we need to consider 
\begin{equation}\label{phi-def}
\begin{aligned}
\varphi(z) & =  \liminf_{x \rightarrow z, \, x\in\Omega}  \max(1, G(x, x_0)), \\ 
\psi(z) & = \limsup_{x \rightarrow z, \, x\in\Omega}  \, \max(1, G(x, x_0)) , 
  \end{aligned}
 \end{equation}
for $z \in \partial \Omega$. 
Note that $\varphi = \psi =1$ at regular boundary points.  The following result is a generalization of Lemma \ref{conv-lemma}, which allows us to control the behavior of $\varphi $ and $\psi$ at irregular points in a uniform domain.  	

\begin{Lemma}\label{conv-lemma-alt}
Suppose $\Omega \subset \R^n$ is a bounded uniform domain, for $n \geq 2$.  
Suppose $\mu$ is a finite Borel measure  with compact support in $\Omega$.    Let $z \in \partial \Omega$. 
Then 
\begin{equation}\label{min-thin3a}
1\le \varphi(z) \le \psi(z) \le \kappa \,
\varphi(z) \le \kappa \, C_1 , \qquad z \in \Omega , 
 \end{equation}
for constants $\kappa$ and $C_1$, where $\kappa$ depends only on $\Omega$ and $C_1$   depends only on ${\rm dist} (x_0, \partial \Omega)$.  Moreover,  for all  $z \in \partial\Omega$, 
\begin{equation} \label{min-thin3}
\begin{aligned}
 \limsup_{x \rightarrow z, \, x\in\Omega} \frac{G \mu (x)}{m(x)}  = \psi(z) M^*\mu(z)
 &  \leq  \kappa \varphi(z) M^*\mu(z) \\ &=  \kappa \liminf_{x \rightarrow z, \, x\in\Omega}  \frac{G \mu (x)}{m(x)} .  
 \end{aligned}
\end{equation}
\end{Lemma}

\begin{proof}  The inequalities $1 \leq \varphi (z) \leq \psi(z)$ are trivial.  The inequality $\psi (z) \leq \kappa \varphi(z)$ follows from inequality (\ref{liminf-limsup}) with $y=x_0$ and the observation that $\max(1, G(x, x_0)) = G(x, x_0)/m(x)$. Since $x\rightarrow z$, 
we may assume that $|x-x_{0} |  \geq c_1$ for any  $c_1 < \text{dist} \, (x_0, \partial \Omega)$, for $x$ close enough to $z$. Then 
 \[
G(x, x_0) \le c(n) \, |x-x_0|^{2-n} \le c(n) \, c_1^{2-n} , 
\] 
where we suppose again that $n\ge 3$ (the case $n=2$ is treated in a similar way).  Hence,  
\[
\psi (z) \le C_1=\max \left(1, c(n) \, [\text{dist} \, (x_0, \partial \Omega)]^{2-n}\right) , \quad 
\text{for all} \, \, z\in \partial \Omega ,
\] 
and consequently \eqref{min-thin3a} holds. 

 To prove \eqref{min-thin3}, 
  note that by  \eqref{min-thin2},  
	\[ \limsup_{x \rightarrow z, \, x\in\Omega} \frac{G \mu (x)}{m(x)} =  \limsup_{x \rightarrow z, \, x\in\Omega} \frac{G(x, x_0)}{m(x)} \, \lim_{x \rightarrow z, \, x\in\Omega} \frac{G \mu (x)}{G(x, x_0)} =  \psi(z) M^* \mu (z)  \]
 and 
\[ \liminf_{x \rightarrow z, \, x\in\Omega} \frac{G \mu (x)}{m(x)} =
    \liminf_{x \rightarrow z, \, x\in\Omega} \frac{G(x, x_0)}{m(x)} \, \lim_{x \rightarrow z, \, x\in\Omega} \frac{G \mu (x)}{G(x, x_0)}  = \varphi(z) M^* \mu (z)  . \]
Hence, \eqref{min-thin3} is immediate from \eqref{min-thin3a}. 
\end{proof}

 \begin{Thm} \label{u_h-exist} Suppose $ \Omega \subset \R^n$ is a bounded uniform domain, $\omega$ is a locally finite Borel measure on $\Omega$, and $h$ is a positive 
 harmonic function in $\Omega$.  Let $x_0 \in \Omega$ be the reference point in the definition of Martin's kernel.   Let $m(x) = \min (1, G(x,x_0))$, and let $\mu_h$ be the Martin's representing  measure for $h$.  

\vspace{0.1in}

(A)  There exists $C>0$ ($C$ depending only on $\Omega$ and $\Vert T \Vert$) such that if $\Vert T \Vert <1$ (equivalently, (\ref{equivnormTless1}) holds with $\beta <1$) and 
\begin{equation} \label{martin-suff}
 \int_{\partial \Omega} e^{C \, \varphi (z)\, M^* (m \omega)(z)} \,   d\mu_h (z)< \infty , 
\end{equation}
then $u_h = \sum_{j=0}^{\infty} T^j h \in L^1_{loc} (\Omega, dx)$ is a positive solution to (\ref{u-harm-int}).   

\vspace{0.1in}

(B) If $u \in L^1_{loc} (\Omega, dx) $ is a positive solution of \eqref{u-harm-int}, then $\Vert T \Vert \leq 1$ and  
\begin{equation} \label{martin-nec}
 \int_{\partial \Omega} e^{\psi(z) \, M^* (m \omega)(z)} \,  d\mu_h(z) < \infty . 
\end{equation}
\end{Thm} 

\begin{proof}
The proof follows the lines of the proof of Theorem \ref{gaugecrit}, so we only sketch the differences. Let $K \subseteq \Omega$ be compact with $|K|>0$. Replacing $Pf$ with $h$, we obtain 
\begin{equation} \label{martin2-alt}
 \int_K u_h (x) \, dx = \int_K h (x) \, dx + \int_{\Omega} u_K (y) h(y) \, d\omega (y)\end{equation}
instead of \eqref{intKuf}.  Using Martin's representation \eqref{martin-rep} instead of 
\eqref{pf-rep-martin}, 
\begin{equation} \label{uKhdomega}
 \int_{\Omega} u_K (y) h(y) \, d\omega (y)  =  \int_{\partial \Omega} \int_{\Omega} M(y,z) u_K (y) \, d \omega (y) \, d \mu_h (z) . 
\end{equation}

For part (A), it suffices to show that $u_K h \in L^1 (\Omega, d\omega)$.  
We claim that 
\begin{equation} \label{Mu0est-alt}
\int_{\Omega} M(y,z) u_K (y) \, d \omega (y) \leq C_K \, e^{C \, \varphi(z) \, M^* (m\omega) (z)}, \quad z \in \partial \Omega,
\end{equation}
which replaces \eqref{Mu0est}, and completes the proof of (A).  To prove \eqref{Mu0est-alt}, we can assume $\omega$ is compactly supported by the exhaustion process above. Choose a sequence of points $x_j$ in $\Omega$ converging to $z$, such that 
\[ \lim_{j \rightarrow \infty} \frac{G(m \omega)(x_j)}{m(x_j)} = \lim_{w \rightarrow z} \inf_{w \in \Omega} \frac{G(m\omega)(w)}{m(w)}.   \]
Then by \eqref{min-thin} with $d \mu= u_K d \omega$, \eqref{GuKomegaest}, and \eqref{min-thin3} with $\mu = m \omega$,
\begin{align*}
  \int_{\Omega} M(y,z) u_K (y) \, d \omega (y) & = \lim_{j \rightarrow \infty} \frac{G(u_K\omega)(x_j)}{G(x_j, x_0)} \\
	& \leq \lim \inf_{j \rightarrow \infty} C_K e^{C G(m\omega)(x_j)/m(x_j)} \\
	& = C_K e^{C \varphi(z) M^* (m\omega) (z) } . 
	\end{align*}

For part (B), equation \eqref{u-harm-int} and Schur's Lemma show that $\Vert T \Vert \leq 1$, as in Theorem \ref{gaugecrit}.  If $u \in L^1_{loc} (\Omega, dx)$, then the  minimal solution $u_h$ also belongs to $L^1_{loc} (\Omega, dx)$ (see the remarks before Definition \ref{qmkernel}).   We claim that the following analogue of \eqref{M-claim} holds:
\begin{equation}  \label{M-claim-alt}
e^{\psi(z) M^* (m \omega) (z)}  \leq  c_K^{-1} \kappa C_1 \, \int_{\Omega} M(y,z) u_K (y) \, d \omega (y) + C_K c_K^{-1},   
\end{equation}
for all $z \in \partial \Omega$, where $C_1\ge 1$ is the constant in \eqref{min-thin3a}, which depends only on $x_0$ and $\Omega$.  Assuming this claim, then \eqref{uKhdomega} implies \eqref{martin-nec} since  $u_K h \in L^1 (\Omega, d \omega)$ by \eqref{martin2-alt}.  To prove \eqref{M-claim-alt}, let $x_j$ be a sequence of points such that 
\[ \lim_{j \rightarrow \infty} \frac{G(m \omega)(x_j)}{m(x_j)} = \lim_{w \rightarrow z} \sup_{w \in \Omega} \frac{G(m\omega)(w)}{m(w)}.   \]
By \eqref{min-thin3} with $d\mu = m\, d\omega$, and recalling that $G(m \omega) = Tm$,
\begin{align*}
  e^{\psi(z) M^* (m \omega) (z)}  &  \leq  \lim_{j \rightarrow \infty} e^{Tm(x_j)/m(x_j)}  \\
	& \leq \lim \sup_{j \rightarrow \infty} C_K^{-1} \frac{G(u_K \omega) (x_j)}{m(x_j)} + C_K c_K^{-1}, 
	\end{align*}
by \eqref{Tm-m} with $c=1$.  By \eqref{min-thin3} with $\mu= u_k \omega$, 
\[ \lim \sup_{j \rightarrow \infty} \frac{G(u_K \omega) (x_j)}{m(x_j)}= \psi (z) M^* (u_k \omega) (z) \leq \kappa C_1 M^* (u_k \omega) (z) ,  \]
which establishes \eqref{M-claim-alt}.
\end{proof}

  \bigskip

 \section{Nonlinear elliptic equations of Riccati type}\label{riccati}

In  this section we treat equation \eqref{nonlineareqn-1}. 
The definition of solutions of \eqref{nonlineareqn-1} is consistent with  our approach 
in the previous sections. 

\begin{Def}\label{defveryweakriccati}   A nonnegative function $v \in W^{1,2}_{loc} (\Omega) $  is a solution of (\ref{nonlineareqn-1}) if 
$v$ is a weak solution in $\Omega$, i.e., 
\begin{equation}\label{weakriccati}
     \int_{\Omega}  \nabla v \cdot \nabla  h \, dx = \int_{\Omega}  |\nabla v|^2 h \, dx + \int_{\Omega} h \,  d\omega, \,\,\, \mbox{for all} \,\,\, h \in C^\infty_0 (\Omega),
\end{equation}
and $v$ has a superharmonic representative (denoted also by $v$) in $\Omega$ whose  
greatest harmonic minorant is the zero function.  
\end{Def}

Since  $v \in W^{1,2}_{loc} (\Omega) $, it is easy to see that \eqref{weakriccati} 
is equivalent to 
\begin{equation}\label{ric-eq-1}
- \triangle v = |\nabla v|^2  + \omega \quad \mbox{in} \, \, \, \, D^{\, \prime}(\Omega),  
\end{equation}
i.e.,  $v$ is 
a distributional solution in $\Omega$. In other words, by the Riesz decomposition theorem (\cite{AG}, Sec. 4.4), 
  $|\nabla v|^2 + \omega$ is the Riesz measure associated with 
 $-\triangle v$, and $v$ satisfies the integral equation 
\begin{equation} \label{integralformmeasure}
v = G (|\nabla v|^2 + \omega) \,\, \hbox{in} \,\, \Omega.
\end{equation}
In bounded Lipschitz domains, \eqref{integralformmeasure} is equivalent to $v$ being a very weak solution of \eqref{nonlineareqn-1} in the sense of \cite{MR}. 

Via the relation $v=\log u$, solutions $v$ of \eqref{nonlineareqn-1} correspond formally to solutions $u$ of   
\eqref{ufeqn} with $f=1$, i.e.,
\begin{equation}\label{dirichlet} 
\left\{ \begin{aligned}
-\triangle u & = \omega \, u, \, \, & u > 0 \quad  &\mbox{in} \, \,  \Omega, \\
u & = 1 \, \, &\mbox{on} \, \,  \partial \Omega.
\end{aligned}
\right.  
\end{equation} 
The minimal solution $u_1$ to \eqref{dirichlet} (the gauge) is given by \eqref{gauge-def}.  

Earlier results on  \eqref{nonlineareqn-1} 
 were obtained in \cite{HMV}, where the problem was posed of   
 finding precise conditions on the boundary behavior of $\omega$  that ensure the existence of solutions. 
 
The precise relation between solutions to \eqref{dirichlet} and \eqref{nonlineareqn-1} is complicated, as discovered by Ferone and Murat (see \cite{FM1}--\cite{FM3} or  Remark 4.2 in \cite{FV2}). In the special case   of smooth domains and absolutely continuous $\omega$, the problem was studied by the authors in  \cite{FV2}, where the condition of the exponential integrability of the balayage of $m \, \omega$ appeared for the first time. In that setup, it was shown that if $u_1$ is the minimal   solution of (\ref{dirichlet}), then $v = \log u_1$ is a solution of (\ref{nonlineareqn-1}).  However, if $v$ is a  solution to (\ref{nonlineareqn-1}) then $u=e^v$ is in general only a supersolution to (\ref{dirichlet}).
  
  In Theorem \ref{riccatithm}, we treat general measures $\omega$ and uniform domains $\Omega$ based on the results of the previous sections. We take this opportunity to give further details on some points in the arguments presented in \cite{FV2}, Sec. 4. We also improve the constant in the exponent of the necessary condition (exponential integrability of the balayage).

\vspace{0.2in}

\begin{proofof} Theorem \ref{riccatithm}.  First suppose that $\Vert T \Vert<1$ and (\ref{martincritsuff-g}) holds with sufficiently  large $C>0$.   By Corollary~\ref{cor}, the Schr\"odinger equation (\ref{dirichlet}) has a positive solution $ u= 1 + \mathcal{G} \omega$.  (This solution was called $u_1$ in the statement of Corollary~\ref{cor}.) Then $u \in L^1_{loc} (\Omega, d\omega)$ and $u$ satisfies the integral equation $u = 1 + G(\omega u)$.  Therefore $u:  \Omega \to [1, +\infty]$  is defined everywhere as a positive superharmonic function in $\Omega$ and hence  is quasi-continuous by the known properties of superhamonic functions. 

In particular, the infinity set $E=\{x\in \Omega: \, u(x)=+\infty\}$ has zero capacity, $\text{cap}(E)=0$, and 
 $u \in  W^{1,p}_{loc}(\Omega)$ when $p< \frac{n}{n-1}$. In fact, $u \in W^{1,2}_{loc}(\Omega)$ as shown in \cite{JMV},  Theorem 6.2, 
but the proof of this stronger property is more involved, and it will not be used below. 

Define $d \mu = -\triangle u = \omega \, u$, where a solution $u \in L^1_{loc}(\Omega, \omega)$ 
to \eqref{dirichlet} 
is understood as in \S \ref{sec2} above. Notice that $u = \frac{d\mu}{d \omega}$  
is the Radon--Nikodym derivative defined $d\omega$-a.e. 
Let $v=\log u$. Then $0\leq  v < +\infty$ $d\omega$-a.e., $v$ is superharmonic in $\Omega$ by Jensen's inequality, and 
 $v  \in W^{1,2}_{loc}(\Omega)$ (see \cite{HKM}, Theorem 7.48; \cite{MZ}, Sec. 2.2). 
 
 We claim that  
  \eqref{weakriccati} holds. We will apply the integration by parts formula 
 \begin{equation}\label{by-parts}
\int_\Omega g \, d \rho = - \langle g, \triangle r \rangle=  \int_\Omega \nabla g \cdot \nabla r \, dx,
\end{equation}
where $g\in  W^{1,2}(\Omega)$ is compactly supported and quasi-continuous in $\Omega$, and $\rho = -\triangle r$ where $r \in W^{1,2}_{loc}(\Omega)$ is superharmonic (see, e.g.,  \cite{MZ}, Theorem 2.39 and Lemma 2.33).  This proof would simplify if we could apply (\ref{by-parts}) with $g = \frac h u, \rho = \mu$, and $r=u$, for $h \in C^{\infty}_0 (\Omega)$.  However, we do not 
use the property $u \in W^{1,2}_{loc}(\Omega)$, so we need an approximation argument.  For $k \in \N$, let 
\[ u_k = \min (u, \, e^k), \quad v_k=\min (v, \, k), \quad \mbox{and} \quad  \mu_k = - \triangle u_k.\]
Clearly $u_k$ and $v_k$ are superharmonic, hence $\mu_k $ is a positive measure.  Moreover, $u_k$ and $v_k$ belong to $W^{1,2}_{loc}(\Omega)\bigcap L^\infty(\Omega)$ (see \cite{HKM}, Corollary 7.20). 

Let $h\in C^\infty_0(\Omega)$. 
We invoke (\ref{by-parts}) with $g=  \frac {h}{u_k}, \rho=\mu_k$, and 
$r=u_k$. Note that $u_k\ge 1$, $g $  is compactly supported since $h$ is, and $ g \in W^{1,2} (\Omega)$ since $u_k\in W^{1,2}_{loc} (\Omega)$ and $h \in W^{1,\infty}(\Omega)$ is compactly supported.  Then by (\ref{by-parts}), we have 
\begin{equation}\label{approx-v_k}
\begin{aligned}
\int_\Omega \frac {h}{u_k} \, d \mu_k & =  \int_{\Omega} \nabla \left(\frac {h}{u_k}\right) \cdot \nabla u_k \, dx\\ 
& =  \int_\Omega \frac {\nabla h}{u_k}  \cdot \nabla u_k \, dx - \int_\Omega \frac {|\nabla u_k|^2}{u_k^2} 
h \, dx \\ 
& =  \int_\Omega \nabla h  \cdot \nabla v_k \, dx - \int_\Omega |\nabla v_k|^2 \,  h \, dx.  
\end{aligned}
\end{equation}

As mentioned above,  $v \in W^{1,2}_{loc} (\Omega)$, and consequently 
$\nabla v_k = \nabla v$ a.e. on $\{ v<k\}$, and $\nabla v_k = 0$ a.e. on $\{ v\ge k\}$ (see \cite{MZ}, Corollary 1.43). Hence, 
\begin{equation*}
\begin{aligned}
\lim_{k \rightarrow \infty} \int_\Omega \nabla h  \cdot \nabla v_k \, dx  &= \int_\Omega \nabla h  \cdot \nabla v \, dx , \\  
\lim_{k \rightarrow \infty} \int_\Omega |\nabla v_k|^2 \,  h \, dx & = \int_\Omega |\nabla v|^2 \,  h \, dx
\end{aligned}
\end{equation*}
by the dominated convergence theorem. 

Since $u$ is superharmonic, $u$ is lower semi-continuous, so the set $\{ x \in \Omega: u(x) > e^k \} \equiv \{u>e^k\}$ is open, and the measure $\mu_k = - \triangle u_k$ is supported on the closed 
set $\{u \le e^k\}$ where $u=u_k$.  Hence  $u=u_k$ $d \mu_k$-a.e., and 
\[
\int_\Omega \frac {h}{u_k} \, d \mu_k=\int_\Omega \frac {h}{u} \, d \mu_k.
\]

We next show that, for any continuous function $h$ with compact support in $\Omega$, 
\begin{equation}\label{claim}
\lim_{k\to \infty} \int_{\Omega} \frac{h}{ u} \, d \mu_k =  \int_\Omega \frac h u \, d \mu. 
\end{equation}
  Without loss of generality 
we assume here that $h\ge 0$. Otherwise we apply the argument below to $h_{+}$ and 
$h_{-}$ separately.

Notice that $u_k \uparrow u$, and consequently $\mu_k \to \mu$ weakly in $\Omega$,
by the weak continuity property (see, for instance, \cite{TW} 
in a rather more general setting), i.e., 
\begin{equation*}
\lim_{k\to \infty} \int_\Omega \phi \, d \mu_k = \int_\Omega \phi \, d \mu
\end{equation*}
for all continuous functions $\phi$ with compact support in $\Omega$. It follows (see \cite{Lan}, Lemma 0.1) that 
\begin{equation}\label{lsc}
\liminf_{k\to \infty} \int_\Omega \phi \, d \mu_k \ge  \int_\Omega \phi \, d \mu
\end{equation}
for all lower semicontinuous  functions $\phi$ with compact support in $\Omega$. The function $\frac{h}{ u}$ is obviously upper semicontinuous   with compact support, so  by \eqref{lsc} applied to $-\frac{h}{ u}$, we deduce 
\begin{equation}\label{upper}
\limsup_{k\to \infty} \int_{\Omega} \frac{h}{ u} \, d \mu_k \le  \int_\Omega \frac h u \, d \mu. 
\end{equation}

To prove an estimate in the opposite direction, we claim that $ \mu_k \ge  \mu$ on the closed 
set $F_k=\{ x \in \Omega: \, u(x)\le e^k\}$. 
It is enough to prove that 
\begin{equation}\label{claim-u_k}
\mu_k (K) \ge \mu(K), \quad \text{for every compact set} \, \, K \subset F_k.
\end{equation}

We verify \eqref{claim-u_k} by using another approximation argument based on a version of Lusin's theorem for certain Green potentials (the so-called semibounded 
potentials, see \cite{Fug}, Sec. 2.6). Notice that $u = G \mu+1$, 
where $d \mu =u \, d \omega$, and $u<\infty$ $d \omega$-a.e., as discussed in  
 \S \ref{sec2}. Moreover, 
$u<\infty$ on $\Omega\setminus\!E$, i.e., outside  
the infinity set $E$, which is obviously a Borel set such that $\mu(E)=0$ 
since $\omega(E)=0$. 

This is also a consequence of the fact that $E$ is a set of zero capacity, 
and $\omega(E)\le \text{cap}(E)$, which follows immediately from \eqref{equivnormTless1}.  In fact, the condition $\mu(E)=0$ is equivalent to absolute 
continuity of $\mu$ with respect to capacity, i.e., 
$\text{cap}(K)=0 \Longrightarrow \mu(K)=0$ for all compact sets $K\subset \Omega$.

Consequently (see \cite{Fug}, Theorem 2.6; \cite{Hel}, Theorem 4.6.3), there exists an increasing  sequence of compactly supported measures $\mu^j$ such that  $u^j=G \mu^j +1\in C(\Omega)$, so  that $\mu^j(K) \uparrow \mu(K)$, for every compact set $K\subset\Omega$, and $G \mu^j\uparrow G \mu$ on $\Omega$, as $j\to \infty$ . It follows that $u^j\uparrow u$, 
and so $\min(u^j, e^k) \uparrow \min(u, e^k)=u_k$  as $j\to \infty$, which yields that the corresponding Riesz measures 
$\mu_k^j$ associated with the superharmonic functions $\min(u^j, e^k)$ have the property 
$\mu_k^j \to \mu_k$ weakly in $\Omega$ as $j\to \infty$. 

Without loss of generality we may assume that actually $u^j(x)<u(x)$ for all $x\in \Omega$. Otherwise 
we replace $u^j $  with $\epsilon_j \, u^j$, where $\epsilon_j\uparrow 1$ is a strictly increasing 
sequence of positive numbers. Then all the properties of $u^j$ remain true. 
 
Obviously, $F_k\subset G^j_k$ where $G^j_k=\{ x \in \Omega: \, u^j(x)< e^k\}$ is an open set 
for every $j, k \in \N$, since $u^j \in C(\Omega)$. Clearly, $u^j=\min(u^j, e^k)$ on $G^j_k$, and so 
$\mu^j$ coincides with $\mu_k^j$ on  $G^j_k$. In particular, $\mu_k^j(K)=\mu^j(K)$ 
for every compact set 
$K \subseteq F_k \subset G_k^j$. 

Since $\mu_k^j \to \mu_k$ weakly, it follows 
 by \eqref{lsc} applied to the lower semicontinuous function 
 $-\chi_K$ that 
 \[
\limsup_{j\to \infty} \mu_k^j(K) \le \mu_k(K). 
\]
 Hence,
 \[
 \mu(K)=\lim_{j \to \infty} \mu^j(K) =\limsup_{j\to \infty} \mu_k^j(K) \le \mu_k(K),  
 \]
 which proves \eqref{claim-u_k}. Consequently, 
 \begin{equation}\label{usc-appl}
 \begin{aligned}
 \liminf_{k \to \infty}  \int_{\Omega} \frac{h}{ u} \, d \mu_k  & \ge \liminf_{k \to \infty} \int_{F_k} \frac{h}{ u} \, d \mu_k 
\\ & \ge \liminf_{k \to \infty} \int_{F_k} \frac{h}{ u} \, d \mu   =  \int_{\Omega\setminus E} \frac{h}{ u} \, d \mu, 
 \end{aligned}
\end{equation}
where $E$ is the infinity set of $u$. As mentioned above,  $\mu(E)=0$, so \eqref{usc-appl} 
actually yields 
\[
 \liminf_{k \to \infty} \int_{\Omega} \frac{h}{ u} \, d \mu_k \ge  \int_{\Omega} \frac{h}{ u} \, d \mu. 
\]
Combining  the preceding inequality with \eqref{upper} proves \eqref{claim}. 

In fact, $\mu_k$ coincides with $\mu$ on the 
set $G_k=\{x\in \Omega: \, \, u(x)<e^k\}$, i.e., 
\begin{equation}\label{claim-great}
\mu_k (K) = \mu(K), \quad \text{for every compact set} \, \, K \subset G_k.
\end{equation}

To prove \eqref{claim-great}, notice  that the set $G_k$ 
is finely open (see \cite{AG}, Sec. 7.1).  
Let  $U_k=\{x\in \Omega: \, \, u(x)>e^k\}$, and $\lambda=\chi_{U_k} \mu$.  
Then clearly $G \lambda\le G \mu=u$ in $\Omega$, and so 
$G \lambda <e^k$ on $G_k$. Moreover,  
$\lambda (G_k)=0$ since $U_k$ and $G_k$ are disjoint. Hence by \cite{Fug}, 
Theorem 8.10, $G \lambda$ is finely harmonic on $G_k$.

 On the other hand, let
  \[
\tilde \mu = \mu_k -\mu|_{F_k}, 
\]
where $\mu_k$ is supported on the closed set $F_k=\Omega\setminus U_k$.  By  \eqref{claim-u_k}, $\tilde \mu$ is a nonnegative measure on $\Omega$. 
Clearly, $G \tilde \mu \le G \mu_k=u_k\le e^k$ in $\Omega$. 
 Since $u_k-u=0$ on $G_k$, it follows that  
\[
G \tilde \mu=u_k-u + G \lambda
\] 
is finely harmonic on $G_k$. 
Hence applying 
  \cite{Fug}, Theorem 8.10 in the opposite direction, we deduce that $\tilde\mu(G_k)=0$, so 
$\tilde \mu(K)=\mu_k(K) - \mu(K) =0$ 
for every compact set $K\subset G_k$. The proof of  \eqref{claim-great} 
is complete.

As noted above, $u = \frac{d\mu}{d \omega}$  
is the Radon--Nikodym derivative defined $d\omega$-a.e., and $\mu(E)=\omega(E)=0$,  where $E= \{ x \in \Omega: u(x)=\infty\}$,
hence
\[
\int_\Omega h  \, d\omega  =  \int_\Omega \frac{h}{u}  \, d \mu
= \lim_{k \to \infty} \,  \int_\Omega \frac {h}{u_k} \, d \mu_k.  
\]
 Passing to the limit as $k \to \infty$ in \eqref{approx-v_k},  we obtain 
\begin{equation*}
\begin{aligned}
\int_\Omega h  \, d\omega & = \int_\Omega \nabla h \cdot \nabla v \, dx - \int_\Omega |\nabla v|^2 \, h \, dx , 
\end{aligned}
\end{equation*}
 for all $h \in C^\infty_0(\Omega)$,
which justifies equation \eqref{weakriccati}. 

By the Riesz decomposition theorem,  
\begin{equation}\label{integral-form}
v = G(-\triangle v) + g =  G(|\nabla v|^2 +\omega) + g,
\end{equation}
where $g$ is the greatest harmonic minorant of $v$.  Since $v \geq 0$, a harmonic minorant of $v$ is $0$, so $g \ge 0$.  It follows from (\ref{integral-form}) and the equation  $u = G(u\omega) + 1$ that 
$$
g \le v=\log u = \log \left (G (u \omega) + 1\right)\le G (u \omega). 
$$
Since $G(u \omega)$ is a Green potential, the greatest harmonic minorant of $G (u\omega)$ is $0$,  therefore  $g=0$.  Hence $v$ is a solution of (\ref{nonlineareqn-1}).  This completes the proof of Theorem \ref{riccatithm} (A).

Conversely, suppose $v\in W^{1,2}_{loc}(\Omega)$ is a solution of equation (\ref{nonlineareqn-1}), that is, $v = G (|\nabla v|^2 + \omega)$.  Then $v \geq 0$ is superharmonic, 
$d\nu = |\nabla v|^2  dx + d\omega$ is the corresponding Riesz measure, 
and 
\eqref{ric-eq-1} holds.  Let  $v_k = \min\, (v, \, k)$ and $\nu_k = -\triangle v_k$, for $k=1,2, \ldots$. Clearly,   $v_k\in W^{1,2}_{loc}(\Omega)\bigcap L^\infty(\Omega)$ is superharmonic. 

Next, as in the proof of \eqref{claim-u_k} above, we observe that 
$\nu_k \ge \nu$ on the set $F_k=\{x\in \Omega: \, v(x)\le k\}$. To verify this claim, it is enough to check that 
\begin{equation}\label{claim-v_k}
\nu_k (K) \ge \nu(K), \quad \text{for every compact set} \, \, K \subseteq F_k.
\end{equation}
The preceding inequality is deduced again using the approximation argument based on \cite{Hel}, Theorem 4.6.3. It requires 
the existence of a Borel set $E\subset \Omega$ such that 
$G \nu<\infty$ on $\Omega\!\setminus\!E$, and $\nu(E)=0$. Let $E=\{x \in \Omega: \, v(x)=\infty\}$. 
Then $E$ is a Borel set and $\text{cap}(E)=0$. 
We need to show that 
$\nu(E)=0$. 

It is known  (see \cite{HMV}, Lemma 2.1) that since $v\in W^{1,2}_{loc}(\Omega)$ is a solution to \eqref{ric-eq-1}, then 
\[
\int_\Omega h^2 d\nu=\int_\Omega |v|^2 h^2 dx + \int_\Omega h^2 d\omega \le 4 \int_\Omega |\nabla h|^2 dx, 
\]
for all $h \in C^\infty_0(\Omega)$. It follows immediately that $\nu(F)\le 4 \, \text{cap}(F)$ 
for all compact (and hence Borel) sets $F$. Since  $\text{cap}(E)=0$, we see that 
$\nu(E)=0$, which completes the proof of \eqref{claim-v_k}. 

We remark that actually $\nu_k = \nu$ on $G_k$,  
where $G_k=\{x \in \Omega: \, v(x)<k\}$, 
exactly as was shown above for $\mu_k = \mu$ on $G_k$ (with $e^k$ in place of $k$). However, we do not need  this fact in the remaining part of the proof.

Since $\nabla v = \nabla v_k$ $dx$-a.e. on $F_k$,  and $\nabla v_k=0$ $dx$-a.e. outside  $F_k$, 
it follows from  \eqref{claim-v_k} that 
\begin{equation}\label{just0}
-\triangle v_k =\nu_k  \ge \chi_{F_k} \nu  = |\nabla v_k|^2 + \chi_{F_k}\, \omega, 
 \end{equation}
as measures. In other words, 
\begin{equation}\label{just1}
-\triangle v_k=\nu_k=|\nabla v_k|^2 + \chi_{F_k}\, \omega  +\lambda_k,
 \end{equation}
where $\lambda_k$ is a nonnegative measure in $\Omega$ supported on $F_k$. 
In fact, as discussed above, $\lambda_k=0$ outside the set $\{x\in \Omega: \, u(x)=k\}$.

Let $u = e^v \geq 1$, $u_k=e^{v_k}$ and $\mu_k=-\triangle u_k$. Clearly, 
$\nabla u_k=\nabla v_k \, e^{v_k}$, so $u_k\in W^{1,2}_{loc}(\Omega)\bigcap L^\infty(\Omega)$.  
 We claim that 
\begin{equation}\label{just2}
\mu_k=-\triangle u_k = -\triangle v_k \, e^{v_k} -|\nabla v_k|^2  \, e^{v_k} \ge 0.  
\end{equation}
To prove  (\ref{just2}), we use integration by parts 
(\ref{by-parts}) with  $g=h e^{v_k}$, where  $h \in C^\infty_0(\Omega)$, and 
 $v_k$ in place of $r$: 
\begin{equation*}
\begin{aligned}
\int_\Omega h \, e^{v_k} \, d \nu_k & = 
\int_\Omega \nabla (h \, e^{v_k}) \cdot \nabla v_k \, dx 
\\ & = \int_\Omega  e^{v_k} \, \nabla h \cdot  \nabla v_k  \, dx  + \int_\Omega  h \, |\nabla v_k|^2 \, e^{v_k} \,dx \\ 
& = \int_\Omega \, \nabla h \cdot  \nabla u_k  \, dx  + \int_\Omega  h \, |\nabla v_k|^2 \, e^{v_k} \,dx 
\\
& 
 =  \int_\Omega h  \, d \mu_k  + \int_\Omega  h \, |\nabla v_k|^2 \, e^{v_k} \,dx. 
\end{aligned}
\end{equation*}
Hence, first applying  \eqref{just2}  and then \eqref{just1}, 
we obtain 
\begin{equation*}
\begin{aligned}
\langle h, \mu_k\rangle &   =\int_\Omega h \,   d \mu_k \\ 
 & = \int_\Omega h \, e^{v_k} \, d \nu_k 
- \int_\Omega    h  \, | \nabla v_k |^2 \, e^{v_k}  \,  dx \\  
& =\int_\Omega h \,  e^{v_k} \  \chi_{F_k}  \, d \omega + \int_\Omega h \,  e^{v_k} \   d \lambda_k \\ 
& =\int_\Omega h \,  e^{v} \  \chi_{F_k}  \, d \omega + \int_\Omega h \,  e^{v} \   d \lambda_k.
\end{aligned}
\end{equation*}
 From the preceding equation it follows that, for all $h\in C^\infty_0(\Omega)$,   $h \ge 0$, 
\begin{equation}\label{mu_k-G_k}
\langle h, \mu_k\rangle \ge \int_\Omega h \,  u \  \chi_{F_k}  \, d \omega  \ge 0.
\end{equation}
Since $v_k$, and hence $u_k$, is lower semicontinuous, 
it follows that $u_k$ is superharmonic in $\Omega$.

Clearly,  
 $u= \lim 
_{k \to +\infty} u_k$ is a superharmonic function in $\Omega$ as the limit of the 
increasing sequence of superharmonic functions $u_k$, since $u=e^v\not\equiv\infty$. 
Moreover, as mentioned above, the infinity set $E$ on which $u=e^v=\infty$
has zero capacity, and 
$\omega(E)\le \nu(E)\le 4 \, \text{cap}(E)$, so 
$\omega(E)=0$.

Since $-\triangle u_k=\mu_k \to \mu$  weakly in $\Omega$, where $\mu = - \triangle u$, passing to the limit as 
$k \to \infty$ in  \eqref{mu_k-G_k} and using the monotone convergence theorem on the right-hand side yields 
\[
\langle h, \mu\rangle \ge \int_{\Omega\setminus E} h \,  u  \, d \omega =\int_{\Omega} h \,  u  \, d \omega \ge 0.
\]

Hence $u$ is superharmonic, and 
\begin{equation}\label{just3}
-\triangle u \ge  \omega \, u \quad \text{in} \, \, \Omega 
\end{equation}
 in the sense of measures. 

It follows from (\ref{just3}) that 
 $\tilde \omega = - \triangle u-\omega u$ is a non-negative measure in $\Omega$,  so by the Riesz decomposition theorem
$$
u = G(-\triangle u) + g =  G(\omega u) + G \tilde \omega+ g \geq G(\omega u) + g,
$$
where $g$ is the greatest harmonic minorant of $u$. Since $u \ge 1$, i.e., 
$1$ is a harmonic minorant of $u$, it follows that $g \ge 1$, and consequently, 
\begin{equation}\label{iter}
u \ge G(\omega u ) + 1 = Tu + 1,
\end{equation}
for $T$ defined by (\ref{defT}).  Since $u \ge Tu$, it follows by Schur's test that 
$||T||_{L^2(\Omega, \omega) \to L^2(\Omega, \omega)} \le 1$, and hence 
 (\ref{equivnormTless1}) 
holds with $\beta =1$.

Iterating (\ref{iter}) and taking the limit, we see that 
$$
\phi \equiv 1 + \mathcal{G} \omega = 1 + \sum_{j=1}^{\infty} G_j \omega= 1 + \sum_{j=1}^{\infty} T^j 1  \le u < +\infty \, \, \text{a.e.},
$$
and 
$$
\phi = G(\omega \phi ) + 1.
$$
Hence $\phi$ is a positive  solution of (\ref{dirichlet}).  Thus (\ref{martincritsuff-g}) holds by Corollary \ref{cor} (B). This completes the proof of  Theorem \ref{riccatithm} (B).
 \end{proofof}

\noindent{\bf Remarks.} 1. As in \cite{FV2} for smooth domains and $\omega \in L^1_{loc} (\Omega)$, 
our sufficiency results   hold in uniform domains for signed measures $\omega$,  if $\omega$ is 
replaced with $|\omega|$ both in the spectral conditions (\ref{normTless1}), (\ref{equivnormTless1}), and
 conditions (\ref{martincritsuff-g}), (\ref{martincritnec-g}). 

2. The lower pointwise estimates of solutions in Theorem \ref{mainufest}(B) 
 are still true for signed measures $\omega$, under some additional assumptions 
 (see  \cite{GV1}). However, the upper pointwise estimates Theorem \ref{mainufest}(A) 
 are no longer true in general, unless we replace 
 $\omega$ with $|\omega|$.  

3. It is still unclear under which (precise) additional assumptions on the quadratic form of $\omega$ 
the main existence results and upper estimates of solutions remain valid. Some results of this type are discussed in \cite{JMV}, but without 
the prescribed boundary conditions.

\end{document}